\documentclass[11pt]{article}

\usepackage{latexsym}
\usepackage{amssymb}
\usepackage{amsthm}
\usepackage{amscd}

\usepackage{amsmath}
\usepackage{mathtools,leftindex,tensor,mhchem}

\usepackage{mathrsfs}
\usepackage[all]{xy}
\usepackage{hyperref} 
\usepackage[usenames,dvipsnames]{color}
\usepackage{graphicx,eepic}
\usepackage{float}

\usepackage{color}
\usepackage{latexsym}
\usepackage{amssymb}
\usepackage{amsthm}
\usepackage{amscd}
\usepackage{amsmath}
\usepackage{mathrsfs}
\usepackage{graphicx}
\usepackage{hyperref}
\usepackage{shuffle}
\usepackage{tikz-cd}
\usepackage[all]{xy}
\input xy \xyoption{frame}
\usepackage{ytableau}

\usepackage{enumerate,tikz}
\usepackage{color}
\usepackage{mathdots}
\usepackage[vcentermath,enableskew]{youngtab}

\usepackage{mathabx}
\usepackage{mathtools}
\usepackage{array}

\definecolor{darkred}{rgb}{0.7,0,0} 
\newcommand{\defn}[1]{{\color{darkred}\emph{#1}}} 

\numberwithin{equation}{section}

\theoremstyle{definition}
\newtheorem* {theorem*}{Theorem}
\newtheorem* {conjecture*}{Conjecture}
\newtheorem{theorem}{Theorem}[section]

\theoremstyle{definition}

\theoremstyle{definition}

\newtheorem* {example*}{Example}

\newtheorem{lemma}[theorem]{Lemma}
\theoremstyle{definition}
\newtheorem{definition}[theorem]{Definition}
\theoremstyle{definition}

\newtheorem{proposition}[theorem]{Proposition}
\newtheorem{corollary}[theorem]{Corollary}

\newtheorem* {remark*}{Remark}
\newtheorem {remark}[theorem]{Remark}
\theoremstyle{definition}
\newtheorem {example}[theorem]{Example}
\theoremstyle{definition}

\theoremstyle{definition}

\theoremstyle{definition}

\theoremstyle{definition}

\def\({\left(}
\def\){\right)}

\newcommand{\QQ}{\mathbb{Q}}

\newcommand{\cK}{\mathcal{K}}

\def\barr{\begin{array}}
\def\earr{\end{array}}
\def\ba{\begin{aligned}}
\def\ea{\end{aligned}}
\def\be{\begin{equation}}
\def\ee{\end{equation}}

\def\qquand{\qquad\text{and}\qquad}

\def\cH{\mathcal H}

\def\ben{\begin{enumerate}}
\def\een{\end{enumerate}}

\def\cF{\mathcal F}

\def\x{\textbf{x}}
\def\y{\textbf{y}}

\def\c{\textbf{c}}
\def\d{\textbf{d}}

\def\R{R}

\def\cF{\mathcal F}

\def\h {\mathrm{ht}}

\def\m{\mathfrak{m}}

\def\invsim_i{\overset{\mathrm{i}}{\underset{\mathrm{inv}}{\sim}}}

\def\m{\mathbf{m}}

\usepackage[colorinlistoftodos]{todonotes}

\newcommand{\mleftrightarrow}[1]{\mathbin{{\overset{#1}{\longleftrightarrow}}}}

\def\R{\mathcal{R}}

\usepackage{fullpage}

\def\ben{\begin{enumerate}}
\def\een{\end{enumerate}}

\def\h{\mathrm{ht}}

\makeatletter
\renewcommand{\@makefnmark}{\mbox{\textsuperscript{}}}
\makeatother

\allowdisplaybreaks[1]

\UseCrayolaColors

\begin{document}
\title{Recursive structures of molecules and cells in Gelfand $S_n$-graphs}
\author{
Zhiqiang DAI \\ Beihang University \\ {\tt daizhiqiang@buaa.edu.cn}
\and
Yifeng ZHANG \\  South China Normal University \\ {\tt calvinz314159@gmail.com}
}

\date{}

\maketitle

\begin{abstract}
$W$-graphs, representing the multiplication action of the standard basis on the canonical basis in the Iwahori-Hecke algebra are introduced by Kazhdan and Lusztig. Marberg defined a generalized $W$-graph, the Gelfand W-graph, corresponding to the Hecke algebra modules instead of Hecke algebras. To classify the molecules and cells of the Gelfand $S_n$-graphs, in this paper, we introduce a recursive structure of $S_n$ and then discuss the action of the recursive structure on the molecules. Using this structure, we prove that the distinguished nonnesting molecule is a cell.
\end{abstract}

\setcounter{tocdepth}{2}
\tableofcontents

\section{Introduction}
For any Coxeter system $(W,S)$, the associated \defn{Iwahori-Hecke algebra} $\cH$  is equipped with two distinct bases: a \defn{standard basis} $\{H_w : w \in W\}$ and \defn{Kazhdan--Lusztig basis} $\{ \underline H_w : w \in W\}$ \cite{KL}. The action of left multiplication by elements of the standard basis on the Kazhdan-Lusztig basis can be visualized as a directed graph, i.e., the \defn{left Kazhdan-Lusztig graphs} of $W$. These graphs serve as canonical examples of \defn{$W$-graphs}, weighted directed graphs that encode $\cH$-module structures with canonical bases analogous to $\{ \underline H_w : w \in W\}$.

A crucial combinatorial challenge in studying $W$-graphs is the classification of their \defn{cells}, defined as the strongly connected components of the graphs. This task is foundational because the cells inherits the $W$-graph structure from the original graph by restriction. A parallel research goal focuses on characterizing \defn{molecules} within W-graphs, which are the connected components of the undirected graph formed by retaining only bidirectional edges.

Among the $W$-graphs, we concentrated on the \defn{Gelfand $W$-graphs} defined by Marberg and Zhang in \cite{MZ}. These graphs have \defn{fixed-point-free(FPF) involutions} as vertices and they are generated by \defn{perfect models}, which is the sum of all irreducible representations without repetition.

In this paper, we focus on the type A case, when $W$ is the symmetric group $S_{n}$ for even $n$. Let $\cF_n$ be the set of FPF involutions in $S_{n}$ for even $n$. In this case, Marberg and Zhang classified the molecules using combinatorial methods in \cite{MZ3}. They defined a Robinson-Schensted-Knuth-like insertion, which maps the FPF involutions to special tableaus. Then the molecules are just the set of FPF involutions corresponding to tableaus with the same shape. 

 The FPF involutions can be partitioned into $n$ subsets according to the other element in the 2-cycle containing 1. The subset of $\cF_n$,  whose elements begin with $(1,2)$, can correspond to $\cF_{n-2}$ by an operator $\rho$. Therefore, we can represent $\cF_n$ by some copies of $\cF_{n-2}$, and call it the \defn{recursive structure} of $\cF_n$. Applying the recursive structure, we analyze its action on the molecules and prove that the distinguished molecule $C_1^n$ is a cell.

This paper is organized as follows. Section~\ref{prelim-sect} contains some preliminaries on $W$-graph and molecules. Section \ref{recurs-sect} proposes the recursive structure and its properties. Section \ref{stair-sect} gives a criterion of the boundaries between two different molecules. In Section~\ref{mole-sect}, we study the actions of recursive structures on the molecules. Finally, in Section~\ref{cell-sect}, we prove that the distinguished nonnesting molecule is a cell, using a two-row Specht quotient, integral Garnir straightening, and the web basis.


\section{Preliminaries}\label{prelim-sect}
In this section, we introduce the quasiparabolic Kazhdan-Lusztig theory which restricts the Coxeter group $W$ to type A.

\subsection{Quasiparabolic conjugacy class of type A}

Rains and Vazirani introduce the following definitions in \cite[\S2]{RV}.

\begin{definition}\label{scaled-def}
A \defn{scaled $W$-set} is a $W$-set $X$ with a height function $\h : X \to \QQ$ satisfying
\[|\h(x) - \h(sx)| \in \{0,1\}\qquad\text{for all $s \in S$ and $x \in X$.}\]
\end{definition}


Denote the set of reflections in $W$ by
$\R = \{ wsw^{-1} : w \in W \text{ and }s \in S\}.$
 
\begin{definition}\label{qp-def} A scaled $W$-set $(X,\h)$ is \defn{quasiparabolic} if both of 
the following properties hold:
\ben
\item[] \hspace{-7mm}(QP1) If $\h(rx) = \h(x)$ for some $(r,x) \in \R\times X$ then $rx =x$.

\item[] \hspace{-7mm}(QP2) If  $\h(rx) > \h(x)$ and $\h(srx) < \h(sx)$ for some $(r,x,s) \in \R\times X \times S$ then $rx=sx$.

\een
\end{definition}

\begin{example}\label{case0-ex}
The set  $W$ with height function $\h=\ell$
is quasiparabolic relative to its action on itself by left (also, by right) multiplication
and also when viewed as a scaled $W\times W$-set relative to the action $(x,y) : w\mapsto xwy^{-1}$;
 see \cite[Theorem 3.1]{RV}.
\end{example}


\begin{example}\label{cc-ex}
A conjugacy class in $W$ is a scaled $W$-set relative to  conjugation  and the  height function $\h=\ell/2$. This scaled $W$-set is sometimes but not always quasiparabolic.
\end{example}

We restate \cite[Corollary 2.13]{RV} as the  lemma which follow this definition:

\begin{definition} 
An element $x $ in a scaled $W$-set  $X$ is \defn{$W$-minimal} (respectively, \defn{$W$-maximal}) if $\h(sx) \geq \h(x)$ (respectively, $\h(sx) \leq \h(x)$) for all $s \in S$.
\end{definition}

\begin{lemma}[Rains and Vazirani \cite{RV}] \label{minimal-lem} 
If a scaled $W$-set is quasiparabolic, then each of its orbits contains at most one $W$-minimal element and at most one $W$-maximal element. These elements, if they exist, have  minimal (respectively, maximal) height in their $W$-orbits.
\end{lemma}

\begin{remark}
This property is enough   to nearly classify the quasiparabolic conjugacy classes in the symmetric group. Assume that $W = S_n$, $S = \{ s_i = (i,i+1) : i =1,\dots,n-1\}$ with the height function $\h = \ell/2$.
 Suppose $\cK \subset S_n$ is a quasiparabolic conjugacy class.  Since $\cK$ is finite, it contains a unique $W$-minimal element by Lemma \ref{minimal-lem}. $\cK$ consists of involutions since every permutation is conjugate in $S_n$ to its inverse. There are $1+\lfloor n/2 \rfloor$ such conjugacy classes: $\{1\}$ and the conjugacy classes of $s_1s_3s_5\cdots s_{2k-1}$ for positive integers $k$ with $2k\leq n$. $\{1\}$ is trivially quasiparabolic, while the conjugacy class of $s_1s_3s_5\cdots s_{2k-1}$ is quasiparabolic only if $2k =n$, since otherwise  $s_2s_4s_6\cdots s_{2k}$ belongs to the same conjugacy class but has the same (minimal) length.
The only remaining conjugacy class, consisting of the fixed-point-free involutions in $S_n$ for $n$ even, is  quasiparabolic by \cite[Theorem 4.6]{RV}. We denote it as $\cF_n$.
\end{remark}

For the rest of this section, $(X,\h)$ denotes  a fixed quasiparabolic $W$-set.
%
%

Then we have this definition from \cite[\S5]{RV}, which attaches to $X$ a certain partial order:

 \begin{definition}\label{bruhat-def}
 The \defn{Bruhat order} on 
 a quasiparabolic $W$-set 
 $X$ is  the weakest partial order $\leq$  with
$ x \leq rx $ for all $x \in X $ and $r\in \R$ with $ \h(x) \leq \h(rx)$.
\end{definition}

It follows immediately from the definition that if $x,y \in X$ then $x < y$ implies $\h(x) < \h(y)$.  Rains and Vazirani develop several other general properties of the Bruhat order in \cite[Section 5]{RV}. 
Among these properties, we only quote the following lemma (which appears as \cite[Lemma 5.7]{RV}) for use later:
 
\begin{lemma}[Rains and Vazirani \cite{RV}]\label{bruhat-lem}
Let $x,y \in X$  such that $x\leq y$ and $s \in S$. Then
 \[sy \leq y \ \Rightarrow\ sx \leq y
 \qquand
 x\leq sx
 \ \Rightarrow\ x \leq sy.
 \]
\end{lemma}

\begin{lemma}[Rains and Vazirani \cite{RV}]\label{reflec-lem}
Let $x \in \cF_n$ and let $r, r'$ be reflections such that $rxr = r'xr' \neq x$. Then $r' \in \{r,xrx \}$.
\end{lemma}

    \subsection{Quasiparabolic Hecke algebra}
    Let $\mathcal{A}=\mathbb{Z}\left[v, v^{-1}\right]$. The Iwahori-Hecke algebra of $(W, S)$ is the $\mathcal{A}$-algebra with basis $\left\{H_w: w \in W\right\}$ satisfying
\begin{equation}
    H_s H_w=\left\{\begin{array}{ll}
H_{s w} & \text { if } \ell(s w)>\ell(w) \\
H_{s w}+\left(v-v^{-1}\right) H_w & \text { if } \ell(s w)<\ell(w)
\end{array} \quad \text { for } s \in S \text { and } w \in W .\right.
\end{equation}

The unit of this algebra is $H_1=1$. There is a unique ring involution of $\mathcal{H}$, written $h \mapsto \bar{h}$ and called the bar operator, such that $\bar{v}=v^{-1}$ and $\overline{H_s}=H_s^{-1}=H_s-\left(v-v^{-1}\right)$ for all $s \in S$. More generally, an $\mathcal{H}$-compatible bar operator for a module $\mathcal{A}$ is a $\mathbb{Z}$-linear map $\mathcal{A} \rightarrow \mathcal{A}$, also written as $a \mapsto \bar{a}$, such that $\overline{h a}=\bar{h} \cdot \bar{a}$ for all $h \in \mathcal{H}$ and $a \in \mathcal{A}$.

As the generalization of $\mathcal{H}$-module, the $\mathcal{A}$-module $\mathcal{M}$ introduced by Marberg \cite{Marberg} is spanned by the basis $\{\mathcal{M}_x,x\in X\}$ satisfying the following  multiplication rule.
\begin{theorem}[\cite{RV}, Thm. 7.1 and \cite{Marberg} Thm. 3.16]
There is a unique $\mathcal{H}$-module structure on $\mathcal{M}$ such that for all $ s \in S $ and $ x \in X $
	\begin{equation}
	    H_s M_x =\left\{
	\begin{aligned}
		&M_{sx}  &\text{\, if \,} \mathrm{ht}(sx)>\mathrm{ht}(x), \\
		&M_{sx}+(v-v^{-1})M_x  &\text{\, if \,} \mathrm{ht}(sx)<\mathrm{ht}(x), \\
		&vM_x   &\text{\, if \,}  \mathrm{ht}(sx)=\mathrm{ht}(x).	
	\end{aligned}
	\right.
	\end{equation}
This $\mathcal{H}$-module have the following additional properties:\\
  (1) $\mathcal{M}$ has a unique $\mathcal{H}$-compatible bar operator with $\overline{M}_{x_1} = M_{x_1}$. \\
  (2) $\mathcal{M}$ has a unique basis $\{\underline{M}_w : w \in \cF_n\}$, with
  \begin{center}
	$\underline{M}_x=\overline{\underline{M}_x} \in \sum_{w<x}v^{-1}\mathbb{Z}[v^{-1}]\cdot M_w$  \end{center}
\end{theorem}
Recall that $ m_{x,y} $ for $x,y \in X $ are the polynomials in $\mathbb{Z}[v^{-1} ]$ such that
\begin{equation}
     \underline{M}_y= \sum_{x \in X}m_{x,y}M_x.
\end{equation}	
 Write  $\mu_\m(x,y)$ for the coefficients of $v^{-1}$ in $m_{x,y}$. Define $\delta_{x,y}=1$ for $x=y$ and $\delta_{x,y}=0$ for $x\neq y$.

    \begin{lemma}[Marberg \cite{Marberg}, Lemma 3.25]\label{Marberg-lem}
	Let $x,y \in X$ with $x<y$. If there exists $s\in S $ with $sy \leq y $ and $sx >x$, then $\mu_\m(x,y)=\delta_{sx,y}$.
	\end{lemma}

Define $\widetilde{m}_{x,y}= v^{\mathrm{ht}(y)-\mathrm{ht}(x)}m_{x,y}  $  for the simpler formula.

    \begin{lemma}[Marberg \cite{Marberg}, Cor. 3.17]
Let $x, y \in X$ and $s \in S$ .\\
If $s y=y$ then $\widetilde{m}_{x, y}=\widetilde{m}_{s x, y}$ and if $s y<y$ then
\begin{equation}
    \widetilde{m}_{x, y}=\widetilde{m}_{s x, y}=\left\{\begin{array}{ll}
\widetilde{m}_{x, s y}+v^2 \cdot \widetilde{m}_{s x, s y} & \text { if } s x>x \\
v^2 \cdot \widetilde{m}_{x, s y}+\widetilde{m}_{s x, s y} & \text { if } s x \leq x
\end{array}\right\}-\sum_{\substack{x<t<s y \\
s t \leq t}} \mu_\m(t, s y) \cdot v^{\mathrm{ht}(y)-\mathrm{ht}(t)} \cdot \widetilde{m}_{x, t} .
\label{rpm}
\end{equation}
\end{lemma}

By definition $m_{x, y}=0$ when $x \not \leq y$. When $x \leq y$, then
$$
v^{\mathrm{ht}(y)-\mathrm{ht}(x)} m_{x, y}=\tilde{m}_{x, y} \in 1+v^2 \mathbb{Z}\left[v^2\right] .
$$
Consequently, $\mu_\m(x, y)=0$ whenever $\operatorname{ht}(y)-\operatorname{ht}(x)$ is even (see [\cite{Marberg}, Proposition 3.18]).
Unlike the Kazhdan-Lusztig case, these integer coefficients can be negative. 
\subsection{\textit{W}-graph}
Following the conventions in \cite{Stembridge}, define a $W$-graph to be a triple $\Gamma=(U, \omega, \tau)$ consisting of a set $U$ with maps $\omega: U \times U \rightarrow \mathbb{Z}\left[v, v^{-1}\right]$ and $\tau: U \rightarrow\{$ subsets of $S\}$ such that the free $\mathbb{Z}\left[v, v^{-1}\right]$-module with basis $\left\{Y_u: u \in U\right\}$ has a left $\mathcal{H}$-module structure in which
$$
H_s Y_u=\left\{\begin{array}{ll}
v Y_u & \text { if } s \notin \tau(u) \\
-v^{-1} Y_u+\sum_{\substack{w \in u \\
s \notin \tau(w)}} \omega(u, w) Y_w & \text { if } s \in \tau(u)
\end{array} \quad \text { for all } s \in S \text { and } u \in U .\right.
$$

We view $\Gamma$ as a weighted digraph with edges $u \stackrel{\omega(u, w)}{\longrightarrow} w$ for each $u, w \in V$ with $\omega(u, w) \neq 0$. Marberg \cite{Marberg} proved that the two triples $\Gamma^\m_n=(V,\omega_\m,\tau_\m)$ is a $W$-graph with integer edge weights. Here $\tau_\m: X \rightarrow \mathcal{P}(S)$ is the map defined by
$$
\tau_\m(x)=\{s \in S: sxs \leq x\},
$$
and $\omega_\m: X \times X \rightarrow \mathbb{Z}$ is the map defined by
$$
\omega_\m(x \rightarrow y)= 
\begin{cases}
\mu_\m(x, y)+\mu_\m(y, x) & \text{if } \tau_\m(x) \nsubseteq \tau_\m(y), \\
0 & \text{if } \tau_\m(x) \subseteq \tau_\m(y).
\end{cases}
$$  

  The Kazhdan-Lusztig left cells \cite{KL} are very important equivalence classes indexed by various standard Young diagrams. When computing left cells one encounters the problem of having to compute a large number of Kazhdan-Lusztig polynomials before any explicit description of
their $W$-graphs can be given. Stembridge \cite{Stembridge} defined the cells as a strongly connected component of \textit{W}-graph, which can
be described combinatorially and can be constructed without calculating Kazhdan-Lusztig
polynomials. Nguyen \cite{Nguyen} proved that this kind of cells is isomorphic to Kazhdan-Lusztig left cells for type \textit{A}. 

\section{Recursive relation}\label{recurs-sect}
It is easy to see that $S_{n-2}$ can be embedded to $S_n$. With this embedding, we can view $\cF_{n-2}$ as a subset of $S_n$ and denote this subset with the same symbol $\cF_{n-2}$. To get a recursive map from $\cF_{n-2}$ to $\cF_{n}$, we define $\theta_n:\cF_n\to \cF_n, z\mapsto w_0zw_0$, where $w_0$ is $S_n$-maximal, and $\rho_n :\cF_{n-2} \rightarrow \cF_n , z\mapsto \theta_n(z s_{n-1})$. When no ambiguity is possible, $\rho_n$ is abbreviated as $\rho$.

\begin{proposition}\label{theta-prop}
The map $\theta_n :\cF_{n} \rightarrow \cF_n,z \mapsto w_0 z w_0 $ is an order-preserving one-to-one mapping.
\end{proposition}
\begin{proof}
The conjugation by the element $w_0$ induces an automorphism preserving all Coxeter structures including Bruhat orders and the length (see \cite{BB}, Proposition 2.3.4). 
\end{proof}

\begin{proposition}\label{rho-prop}
Let $Y_1=\rho(\cF_{n-2})$. The map $\rho :\cF_{n-2} \rightarrow Y_1,z \mapsto w_0 z s_{n-1} w_0 $ is an order-preserving one-to-one mapping.
\end{proposition}
\begin{proof}
Since $s_{n-1}$ commutes with all elements of $\cF_{n-2}$, the map $\rho$ has the same properties as $\theta_n$.
\end{proof}

Define
\[
\sigma_j=s_js_{j-1}\cdots s_2s_1\qquad(1\leq j\leq n-1).
\]
Every element of $Y_1$ contains the cycle $(1,2)$, so conjugation by
$\sigma_1=s_1$ fixes $Y_1$. For $1\leq j\leq n-1$, set
\[
Y_j=\sigma_jY_1\sigma_j^{-1}.
\]
For $2\leq j\leq n-1$, this gives $Y_j=s_jY_{j-1}s_j$.
\begin{theorem}\label{recurs-thm}
For $1\leq i\leq n-1$, $Y_i$ is the set of elements $z\in\cF_n$
such that $z(1)=i+1$. Consequently,
\[
\cF_n=Y_1\sqcup Y_2\sqcup\cdots\sqcup Y_{n-1}.
\]
\end{theorem}
We call this the recursive decomposition of $\cF_n$.
\begin{proof}
If $y\in\cF_{n-2}$, then the definition of $\rho_n$ gives
$\rho_n(y)(1)=2$. Hence every element of $Y_1$ has first value $2$.
Conjugation by $\sigma_i$ relabels $2$ as $i+1$, so every element of
$Y_i$ has first value $i+1$.

The sets $Y_i$ are therefore pairwise disjoint. Each has cardinality
$|\cF_{n-2}|=(n-3)!!$, while
\[
\sum_{i=1}^{n-1}|Y_i|=(n-1)(n-3)!!=(n-1)!!=|\cF_n|.
\]
Thus the disjoint union of the $Y_i$ exhausts $\cF_n$, proving the
reverse inclusion as well.
\end{proof}

\begin{example}
For $\cF_6$, we have $Y_1=\{s_1s_3s_5,s_1s_4s_5s_3s_4,s_1s_5s_4s_5s_3s_4s_5\}$. Then we see that for $z\in Y_1$, $z(1)=2$. Similarly, $Y_2=s_2Y_1s_2=\{s_2s_1s_3s_2s_5,s_2s_4s_1s_3s_5s_4s_2,s_2s_1s_5s_4s_5s_3s_4s_5s_2\}, \cdots, Y_5=\{s_1s_2s_3s_4s_5s_1s_2s_3s_4s_1s_2s_3s_1s_2s_1,s_1s_2s_3s_4s_5s_2s_3s_4s_1s_2s_3s_2s_1,s_1s_2s_3s_4s_5s_3s_4s_3s_1s_2s_1\}$. Then $\cF_6=Y_1\sqcup Y_2\sqcup Y_3\sqcup Y_4\sqcup Y_5$.
\end{example}


For $1\leq j\leq n-1$, let $\nu_j$ denote conjugation by $\sigma_j$
on $\cF_n$:
\[
\nu_j(w)=\sigma_jw\sigma_j^{-1}.
\]
Its restriction maps $Y_1$ bijectively to $Y_j$. On $Y_1$,
$\nu_1$ is the identity, and for $2\leq j\leq n-1$,
\[
\nu_j(x)=s_j\nu_{j-1}(x)s_j.
\]
\begin{lemma}\label{nu-lem}
For $2\leq i\leq j$,
\[
\sigma_js_i=s_{i-1}\sigma_j.
\]
Moreover,
\[
\sigma_js_{j+1}=(s_js_{j+1}s_j)\sigma_j,
\qquad
\sigma_js_i=s_i\sigma_j\quad(i\geq j+2).
\]
Consequently, for $w\in\cF_n$,
\[
\nu_j(s_iws_i)=
\begin{cases}
s_{i-1}\nu_j(w)s_{i-1},&2\leq i\leq j,\\
(s_js_{j+1}s_j)\nu_j(w)(s_js_{j+1}s_j),&i=j+1,\\
s_i\nu_j(w)s_i,&i\geq j+2.
\end{cases}
\]
\end{lemma}
\begin{proof}
The first identity follows by moving $s_i$ to the left through
$s_1,\ldots,s_j$ and using the braid relation at $s_{i-1},s_i$.
The second identity is the braid relation
$s_js_{j+1}s_j=s_{j+1}s_js_{j+1}$, and the last follows from
commutation. Conjugating $s_iws_i$ by $\sigma_j$ gives the displayed
three cases.
\end{proof}

\begin{proposition}\label{nuj-prop}
For $1\leq j\leq n-1$, the restriction
$\nu_j:Y_1\to Y_j$ is an order-preserving bijection.
\end{proposition}
\begin{proof}
Bijectivity follows from the definition of $Y_j$. For $x\in Y_1$, a
direct length calculation gives
\[
\operatorname{ht}(\nu_j(x))=\operatorname{ht}(x)+j-1.
\]
Suppose that $x\lessdot y$ in $Y_1$. Then $y=rxr$ for a reflection
$r$, and $\operatorname{ht}(y)=\operatorname{ht}(x)+1$. Setting
$r'=\sigma_jr\sigma_j^{-1}$ gives
\[
\nu_j(y)=r'\nu_j(x)r',
\qquad
\operatorname{ht}(\nu_j(y))
 =\operatorname{ht}(\nu_j(x))+1.
\]
Thus $\nu_j(x)\lessdot\nu_j(y)$. Applying this to every cover in a
saturated chain proves that $\nu_j$ is order preserving.
\end{proof}

\begin{corollary}\label{nu-cor}
For $x\in Y_1$ and $1\leq i\leq j\leq n-1$, we have
$\nu_i(x)\leq\nu_j(x)$.
\end{corollary}
\begin{proof}
For $2\leq t\leq n-1$,
\[
\nu_t(x)=s_t\nu_{t-1}(x)s_t
\quad\text{and}\quad
\operatorname{ht}(\nu_t(x))
 =\operatorname{ht}(\nu_{t-1}(x))+1.
\]
Hence $\nu_{t-1}(x)\lessdot\nu_t(x)$. Chaining these covers from $i$
to $j$ proves the assertion.
\end{proof}

\begin{remark}
Proposition~\ref{nuj-prop} not only proves that each $Y_i$ possesses the same Bruhat order but also illustrates how the reflection $r$ changes under the action of $\nu_j$. This will prove to be extremely useful when classifying molecules in Section~\ref{mole-sect}.
\end{remark}

\begin{proposition}\label{skip prop}
Assume that $y\in Y_i$, $s\in S$, and $sys\notin Y_i$.
\begin{itemize}
\item[(a)] If $i=1$, then $s=s_2$.
\item[(b)] If $2\leq i\leq n-2$, then
$s\in\{s_1,s_i,s_{i+1}\}$.
\item[(c)] If $i=n-1$, then $s\in\{s_1,s_{n-1}\}$.
\end{itemize}
\end{proposition}
\begin{proof}
Since $y(1)=i+1$, conjugation by $s_t$ leaves the value paired with
$1$ unchanged unless $t=1$, $t=i$, or $t=i+1$. For $i=1$,
conjugation by $s_1$ fixes the cycle $(1,2)$, leaving only $s_2$.
For $i=n-1$, the generator $s_n$ does not exist, leaving only
$s_1$ and $s_{n-1}$. This gives the three cases.
\end{proof}

Before we move to the next section, we now label the elements of $\cF_n$ recursively as follows:
First assume we have already labeled elements for $\cF_{n-2}$ as $x_1,x_2\cdots, x_{(n-3)!!}$, then we label $\rho(x_i)\in Y_1$ as $x_i$. Next, for $x_i\in Y_1$, we label $\nu_j(x_i)$ as $x_{i+(j-1)(n-3)!!}$.

\section{Bidirected edges}\label{bidiedge-sect}
The following processes and results are about $\Gamma^\m_n$. Let $x, y \in \cF_n$ with $x <y $. Then $\omega_\m(x \rightarrow y) \neq 0$ and $\omega_\m(y \rightarrow x) \neq 0$ indicate that it is bidirected edge between $x$ and $y$.  Recall that $\tau_\m(x)=\{s\in S: sxs\le x\}$.
\begin{definition}
The \textit{molecules} are subgraphs of the $W$-graph, which obtained by retaining only the bidirected edges. 
\end{definition}
 
\begin{lemma}\label{bidi-lem}
There is a bidirected edge between the vertices $x\le y$ if and only if the following conditions hold:
\begin{itemize}
\item{BE1} There exist $s\in S$ such that $y=sxs$.
\item{BE2} $\tau_\m(x) \nsubseteq \tau_\m(y)$.
\end{itemize}
\end{lemma}
\begin{proof}
This follows directly from the definition of bidirected edges and Lemma~\ref{Marberg-lem}.
\end{proof}
 
We denote this bidirected edge as $x \mleftrightarrow{s} y$ or $x \leftrightarrow y$. Moreover, $x \sim y$ denotes that $x$ and $y$ are in the same molecule. Then we discuss how the map $\rho$ affect the bidirected edges.
\begin{proposition}\label{rho-bi-prop}
For $x\leq y$ in $\cF_{n-2}$, there is a bidirected edge
$x\leftrightarrow y$ if and only if there is a bidirected edge
$\rho(x)\leftrightarrow\rho(y)$. In particular,
$x\sim y$ if and only if $\rho(x)\sim\rho(y)$.
\end{proposition}
\begin{proof}
For $1\leq i\leq n-3$,
\[
\rho(s_ixs_i)=s_{n-i}\rho(x)s_{n-i}.
\]
The forced cycle $(1,2)$ in $\rho(x)$ also gives
\[
\tau_\m(\rho(x))
 =\{s_1\}\cup
 \{s_{n-i}:s_i\in\tau_\m(x)\}.
\]
Thus $\rho$ preserves condition \textit{BE1}, and it preserves both
inclusions and non-inclusions between the corresponding
$\tau_\m$-sets. Lemma~\ref{bidi-lem} now gives the equivalence of
bidirected edges. Applying this equivalence edge by edge to a path,
and using the injectivity of $\rho$, gives the assertion about
molecules.
\end{proof}

For $1 \leq j \leq n-1$, let $\tau_j(x)= \tau_\m(\nu_{j} (x))$ where $x \in Y_1$. To discuss how the map $\nu_j$ act on the set $\tau_\m$, we have the following lemma:
\begin{lemma}\label{tau-lem1}
For $2\leq j\leq n-2$,
$\tau_{j-1}(x)\subset\tau_j(x)$ if and only if
$s_j\in\tau_\m(x)$ and $s_{j+1}\notin\tau_\m(x)$. Moreover,
$\tau_{n-2}(x)\subset\tau_{n-1}(x)$ if and only if
$s_{n-1}\in\tau_\m(x)$.
\end{lemma}
\begin{proof}
Write $u_t=\nu_t(x)$. Since $x\in Y_1$, it contains the cycle
$(1,2)$. Hence $u_{j-1}$ contains $(1,j)$ and $u_j$ contains
$(1,j+1)$. The descent criterion for fixed-point-free involutions
gives
\[
s_{j-1}\in\tau_\m(u_{j-1}),\quad
s_j\notin\tau_\m(u_{j-1}),\quad
s_j\in\tau_\m(u_j),\quad
s_{j+1}\notin\tau_\m(u_j).
\]
The relabeling by $\sigma_j$ also gives
\[
s_{j-1}\in\tau_\m(u_j)
 \Longleftrightarrow s_j\in\tau_\m(x),
\qquad
s_{j+1}\in\tau_\m(u_{j-1})
 \Longleftrightarrow s_{j+1}\in\tau_\m(x).
\]
All other generators have the same membership in
$\tau_\m(u_{j-1})$ and $\tau_\m(u_j)$. Therefore every descent of
$u_{j-1}$ is a descent of $u_j$ exactly when the two conditions in
the statement hold. The inclusion is strict because
$s_j\notin\tau_\m(u_{j-1})$ and $s_j\in\tau_\m(u_j)$.

For the terminal pair, the same relabeling comparison between
$\nu_{n-2}(x)$ and $\nu_{n-1}(x)$ shows that the only possible
obstruction to inclusion is the descent $s_{n-1}$ of $x$. This gives
the final equivalence.
\end{proof}

\begin{corollary}\label{tau-cor}
For $x\in Y_1$ and $2 \leq j \leq n-2$, there is a bidirected edge $\nu_{j-1}(x) \leftrightarrow \nu_{j}(x)$ if and only if $s_{j+1} \in \tau_\m(x)$ or $s_j \notin \tau_\m(x)$. Moreover, $\nu_{n-2}(x) \leftrightarrow \nu_{n-1}(x)$ if and only if $s_{n-1} \notin \tau_\m(x)$.
\end{corollary}
\begin{proof}
Since $\nu_j(x)=s_j \nu_{j-1}(x) s_j$, so we satisfied the condition \textit{BE1}. By the discussion in the proof of Lemma~\ref{tau-lem1}, we see that $\tau_{j}(x)\nsubseteq\tau_{j-1}(x)$ since $s_{j}\not\in \tau_{j-1}(x)$ and $s_{j}\in \tau_{j}(x)$. According to Lemma~\ref{tau-lem1}, $\tau_{j-1}(x) \subset \tau_j(x)$ if and only if $s_j \in \tau_\m(x)$ and $s_{j+1} \notin \tau_\m(x)$. Hence $\tau_{j-1}(x) \nsubseteq \tau_j(x)$ if and only if $s_{j+1} \in \tau_\m(x)$ or $s_j \notin \tau_\m(x)$.

The case when $\nu_{n-2}(x) \leftrightarrow \nu_{n-1}(x)$ is proved similar.
\end{proof}

\begin{theorem}\label{nu-2-thm}
If $x,y \in Y_1$, $x \sim y$ if and only if $\nu_2(x) \sim \nu_2(y)$.
\end{theorem}
\begin{proof}
For every $x\in Y_1$, Theorem~\ref{recurs-thm} gives $x(1)=2$.
Since $x$ is a fixed-point-free involution, this implies $x(2)=1$ and
$x(3)\geq4$. The usual descent criterion therefore gives
$s_2\notin\tau_\m(x)$. Corollary~\ref{tau-cor}, with $j=2$, now gives a
bidirected edge
\[
x=\nu_1(x)\leftrightarrow\nu_2(x).
\]
Thus $x\sim\nu_2(x)$ for every $x\in Y_1$. If $x\sim y$, transitivity
gives
\[
\nu_2(x)\sim x\sim y\sim\nu_2(y).
\]
Conversely, if $\nu_2(x)\sim\nu_2(y)$, then
\[
x\sim\nu_2(x)\sim\nu_2(y)\sim y.
\]
This proves both implications.
\end{proof}

\begin{example}
Recall the $\cF_n$-minimal element $x_1=s_1s_3\cdots s_{n-1} \in Y_1$. We calculate that $\tau_\m(x_1)=\{s_1,s_3,\cdots,s_{n-1}\}$. Then we can obtain a directed path from $x_1$ to $\nu_{n-1}(x_1)$ by Corollary~\ref{tau-cor}, which is
    $$ x_1 \leftrightarrow \nu_2(x_1) \leftarrow \nu_3(x_1) \leftrightarrow \nu_4(x_1) \leftarrow \cdots \leftarrow \nu_{n-2}(x_1) \leftrightarrow \nu_{n-1}(x_1).$$
\end{example}

\begin{theorem}\label{nu-n-thm}
Let $x,y\in Y_1$. If $x<y$, then there is a bidirected edge
$x\leftrightarrow y$ if and only if there is a bidirected edge
$\nu_{n-1}(x)\leftrightarrow\nu_{n-1}(y)$. In particular, for
arbitrary $x,y\in Y_1$,
\[
x\sim y\quad\Longleftrightarrow\quad
\nu_{n-1}(x)\sim\nu_{n-1}(y).
\]
\end{theorem}
\begin{proof}
Every $z\in Y_1$ contains $(1,2)$, so
$s_1\in\tau_\m(z)$ and $s_2\notin\tau_\m(z)$. Directly from the
relabeling by $\sigma_{n-1}$,
\[
\tau_\m(\nu_{n-1}(z))
 =\{s_1,s_{n-1}\}\cup
 \{s_{i-1}:3\leq i\leq n-1,\ s_i\in\tau_\m(z)\}.
\]
Consequently, inclusion or non-inclusion of the $\tau_\m$-sets of
two elements of $Y_1$ is preserved by $\nu_{n-1}$.

A nontrivial simple conjugation between two elements of $Y_1$ must
use some $s_i$ with $3\leq i\leq n-1$, and Lemma~\ref{nu-lem} sends
it to conjugation by $s_{i-1}$. Conversely, a simple conjugation
between two elements of $Y_{n-1}$ that remains in $Y_{n-1}$ uses
$s_i$ with $2\leq i\leq n-2$ and pulls back to $s_{i+1}$. Thus
condition \textit{BE1} is equivalent on the two sides, while the
displayed formula makes condition \textit{BE2} equivalent.
Lemma~\ref{bidi-lem} proves the assertion about edges. Applying the
edge equivalence along paths proves the assertion about molecules.
\end{proof}


The following result shows that each molecule meets the sequence
$\nu_1(x),\nu_2(x),\ldots,\nu_{n-1}(x)$ in an interval.
\begin{theorem}\label{between-mole-thm}
For $x \in Y_1$ and $1\le i<j\le n-1$, if $\nu_i(x)\sim \nu_j(x)$, then $\nu_k(x)\sim \nu_i(x)$ for $i\le k\le j$.
\end{theorem}
\begin{proof}
For partitions of the same integer, write $\lambda\unrhd\mu$ for
dominance order. We first recall the following consequence of Greene's
theorem: if $w\in S_n$ and $s\in S$ satisfy
$\ell(ws)=\ell(w)+1$, then
\begin{equation}\label{weak-shape-monotonicity}
 \operatorname{sh}(w)\unrhd\operatorname{sh}(ws).
\end{equation}
Indeed, the one-line word of $ws$ is obtained from that of $w$ by changing
two adjacent entries $A<B$ into $B,A$. For $r\geq1$, let $G_r(v)$ be the
maximum size of a union of $r$ increasing subsequences of the one-line word
of $v$. Take such a union in the word containing $B,A$, color its entries
with $r$ colors so that every color class is increasing, and move each
selected value among $A,B$ together with its color when the two adjacent
positions are swapped. If both values are selected, they have different
colors since $B>A$. Hence no selected value crosses another value of its own
color, so all color classes remain increasing. Thus
$G_r(w)\geq G_r(ws)$ for every $r$. Greene's theorem identifies these
numbers with the partial row sums of the corresponding
Robinson--Schensted shapes, proving \eqref{weak-shape-monotonicity}.
Taking inverses also shows that
\begin{equation}\label{left-weak-shape-monotonicity}
 \ell(sw)=\ell(w)+1
 \quad\Longrightarrow\quad
 \operatorname{sh}(w)\unrhd\operatorname{sh}(sw),
\end{equation}
since $\operatorname{sh}(v^{-1})=\operatorname{sh}(v)$.

Set $u_t=\nu_t(x)$. We claim that
\begin{equation}\label{nu-shape-chain}
 \operatorname{sh}(u_t)\unrhd\operatorname{sh}(u_{t+1})
 \qquad(1\leq t\leq n-2).
\end{equation}
Let $s=s_{t+1}$. By definition, $u_{t+1}=su_ts$. Since $u_t\in Y_t$,
the cycle of $u_t$ containing $1$ is $(1,t+1)$, whence
\[
 u_t(t+1)=1<u_t(t+2).
\]
The usual length criterion gives $\ell(u_ts)=\ell(u_t)+1$. Moreover,
$(u_ts)^{-1}=su_t$ and
\[
 (su_t)(t+1)=1<(su_t)(t+2).
\]
The second value is greater than $1$ because $u_t$ is fixed-point-free and
$t+1$ is already paired with $1$. Hence
\[
 \ell(su_ts)=\ell((u_ts)^{-1}s)=\ell(u_ts)+1.
\]
Applying \eqref{weak-shape-monotonicity} and
\eqref{left-weak-shape-monotonicity} gives
\[
 \operatorname{sh}(u_t)
 \unrhd\operatorname{sh}(u_ts)
 \unrhd\operatorname{sh}(su_ts)
 =\operatorname{sh}(u_{t+1}),
\]
which proves \eqref{nu-shape-chain}.

Row Beissinger insertion agrees with ordinary Robinson--Schensted insertion
on involutions \cite[Theorem~3.1]{Beissinger}. By
Theorem~\ref{shape-molecule-class}, the hypothesis $u_i\sim u_j$ implies
\[
 \operatorname{sh}(u_i)=\operatorname{sh}(u_j).
\]
For every $i\leq k\leq j$, repeated application of
\eqref{nu-shape-chain} therefore gives
\[
 \operatorname{sh}(u_i)
 \unrhd\operatorname{sh}(u_k)
 \unrhd\operatorname{sh}(u_j)
 =\operatorname{sh}(u_i).
\]
Antisymmetry of dominance order forces
$\operatorname{sh}(u_k)=\operatorname{sh}(u_i)$. Another application of
Theorem~\ref{shape-molecule-class} gives $u_k\sim u_i$, as required.
\end{proof}

\section{Staircase lines}\label{stair-sect}


To figure out the elements of $Y_1$, define a map
\[
\phi: \cF_{n-4} \mapsto \cF_n, x\mapsto \rho_n \rho_{n-2}(x).
\]
Then $\phi(\cF_{n-4}) \subset Y_1$.  Let $\cF_{n-4}=Z_1 \bigsqcup \cdots \bigsqcup Z_{n-5}$ is the recursive decomposition of $\cF_{n-4}$. The map $\phi$ is also one-to-one order-preserving as $\rho$ and it maps $s_i$ to $ s_{i+2}$. Thus we have
\begin{equation}\label{eq1}
\phi(\cF_{n-4})=\phi(Z_1) \bigsqcup \cdots \bigsqcup \phi(Z_{n-5}).
\end{equation}

\begin{proposition}\label{vert-stair-prop}
Let $x\in\phi(\cF_{n-4})$ and $3\leq k\leq n-3$. Then
\[
s_1\nu_k(x)s_1=s_k\nu_k(x)s_k
\quad\Longleftrightarrow\quad
x\in\phi(Z_{k-2}).
\]
If $x\in\phi(Z_{k-2})$, then
\[
\tau_\m(\nu_{k-1}(x))
 \subset\tau_\m(\nu_k(x)),
\]
so $\nu_{k-1}(x)$ and $\nu_k(x)$ are not connected by a bidirected
edge.
\end{proposition}
\begin{proof}
Write $x=\phi(z)$ with $z\in Z_r$. The map $\phi$ adjoins the cycles
$(1,2)$ and $(n-1,n)$ and relabels every entry of $z$ by adding $2$.
Since $z(1)=r+1$, the cycle of $x$ containing $3$ is
\[
(3,r+3).
\]
Set $u=\nu_k(x)$. Conjugation by $\sigma_k$ sends $1$ to $k+1$,
sends $a$ to $a-1$ for $2\leq a\leq k+1$, and fixes larger labels.
Hence
\[
u(1)=k+1,\qquad
u(2)=k\quad\Longleftrightarrow\quad x(3)=k+1.
\]
By Lemma~\ref{reflec-lem},
\[
s_1us_1=s_kus_k
\quad\Longleftrightarrow\quad
us_1u=s_k,
\]
and the latter equality is equivalent to $u(2)=k$. Therefore it
holds exactly when $r=k-2$, which is equivalent to
$x\in\phi(Z_{k-2})$.

In this case $x$ contains the cycle $(3,k+1)$. Thus
$s_k\in\tau_\m(x)$ and $s_{k+1}\notin\tau_\m(x)$. The inclusion now
follows from Lemma~\ref{tau-lem1}. It is strict, so
Lemma~\ref{bidi-lem} excludes a bidirected edge.
\end{proof}

\begin{proposition}\label{hori-stair-prop}
Let $3\leq k\leq n-4$ and let $x=\phi(z)\in\phi(Z_{k-2})$. Assume that
$z(2)\neq k$, or equivalently $x(k+2)\neq4$. Put
\[
y=s_{k+1}xs_{k+1}.
\]
Then $y\in\phi(Z_{k-1})$, and $\nu_k(x)$ and $\nu_k(y)$ are not connected by a bidirected edge.
\end{proposition}
\begin{proof}
Since $z\in Z_{k-2}$, we have $z(1)=k-1$. Hence
$s_{k-1}zs_{k-1}\in Z_{k-1}$. The map $\phi$ sends $s_i$ to
$s_{i+2}$, so
\[
y=s_{k+1}xs_{k+1}=\phi(s_{k-1}zs_{k-1})\in\phi(Z_{k-1}).
\]

Set $u=\nu_k(x)$ and $v=\nu_k(y)$. By Lemma~\ref{nu-lem},
\[
v=(s_ks_{k+1}s_k)u(s_ks_{k+1}s_k).
\]
If $v=s_ju s_j$ for some simple reflection $s_j$, then
Lemma~\ref{reflec-lem} implies that either
$s_j=s_ks_{k+1}s_k$ or
$s_j=u(s_ks_{k+1}s_k)u$. The first reflection is not simple. For the
second one, we have
\[
u(s_ks_{k+1}s_k)u=\sigma_kxs_{k+1}x\sigma_k^{-1}.
\]
Since $x\in\phi(Z_{k-2})$, we have $x(k+1)=3$, and therefore
$xs_{k+1}x$ is the transposition $(3,x(k+2))$. Conjugation by
$\sigma_k$ sends $3$ to $2$. Under the present assumptions the resulting
transposition can be simple only if $x(k+2)=4$, which has been excluded.
Thus condition \textit{BE1} in Lemma~\ref{bidi-lem} fails, and the two
vertices are not connected by a bidirected edge.
\end{proof}

\begin{remark}\label{hori-stair-exception}
The condition $x(k+2)\neq4$ cannot simply be omitted. For example, let
$n=10$, $k=4$, and
\[
x=(1,2)(3,5)(4,6)(7,8)(9,10)\in\phi(Z_2).
\]
Then $x(6)=4$, and $\nu_4(x)$ and
$\nu_4(s_5xs_5)$ are connected by a bidirected edge labelled by $s_2$.
\end{remark}

Propositions~\ref{vert-stair-prop} and~\ref{hori-stair-prop} give the
generic segments of the staircase shown in Figure~\ref{staircase-line1}.
The figure is schematic: the horizontal segments excluded in
Remark~\ref{hori-stair-exception} require a separate analysis.

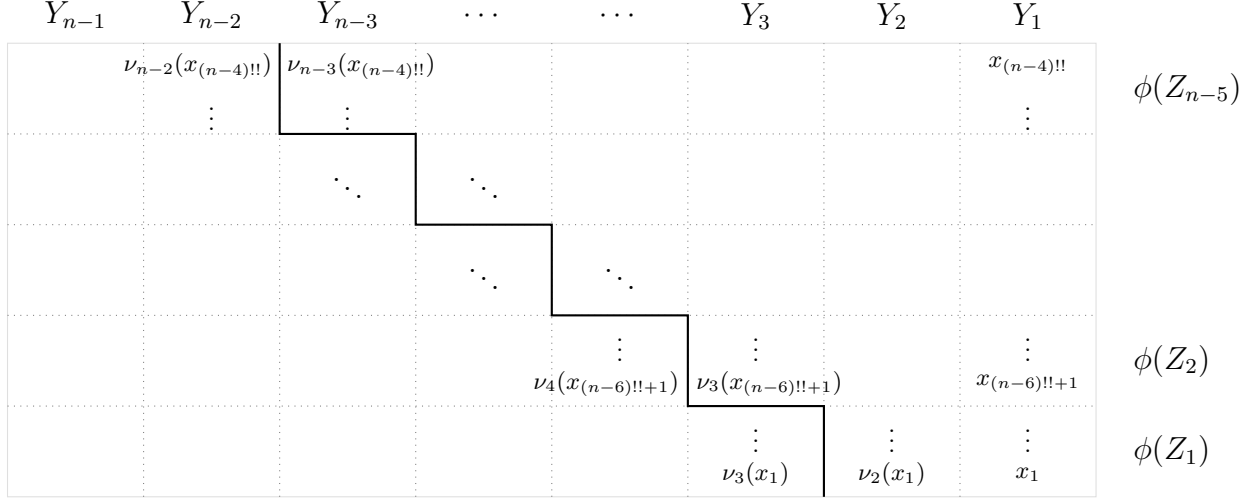
\begin{figure} 
\begin{tikzpicture}[x=1.8cm,y=1.2cm] 

    \draw[lightgray!50] (0,0) rectangle (8, 5);
    \foreach \i in {1,...,7} {
        \draw[dotted, gray] (\i, 0) -- (\i, 5);
    }
    \foreach \j in {1,...,4} {
        \draw[dotted, gray] (0, \j) -- (8, \j);
    }

    \draw[thick, black] (2, 5) -- (2, 4) -- (3, 4) -- (3, 3) -- (4, 3) -- (4, 2) -- (5, 2) -- (5, 1) -- (6, 1) -- (6, 0);

    \foreach \i/\label in {0.5/Y_{n-1}, 1.5/Y_{n-2}, 2.5/Y_{n-3}, 3.5/\cdots, 4.5/\cdots, 5.5/Y_3, 6.5/Y_2, 7.5/Y_1} {
        \node[font=\large] at (\i, 5.3) {$\label$};
    }

    \node[anchor=west, font=\large] at (8.2, 4.5) {$\phi(Z_{n-5})$};
    \node[anchor=west, font=\large] at (8.2, 1.5) {$\phi(Z_2)$};
    \node[anchor=west, font=\large] at (8.2, 0.5) {$\phi(Z_1)$};

    \begin{scope}[every node/.style={font=\footnotesize, inner sep=2pt}]

        \node at (1.4, 4.75) {$\nu_{n-2}(x_{(n-4)!!})$};
        \node at (1.5, 4.25) {$\vdots$};

        \node at (2.6, 4.75) {$\nu_{n-3}(x_{(n-4)!!})$};
        \node at (2.5, 4.25) {$\vdots$};

        \node at (7.5, 4.75) {$x_{(n-4)!!}$};
        \node at (7.5, 4.25) {$\vdots$};

        \node[font=\normalsize] at (2.5, 3.5) {$\ddots$};
        \node[font=\normalsize] at (3.5, 3.5) {$\ddots$};
        \node[font=\normalsize] at (3.5, 2.5) {$\ddots$};
        \node[font=\normalsize] at (4.5, 2.5) {$\ddots$};

        \node at (4.5, 1.7) {$\vdots$};
        \node at (4.4, 1.25) {$\nu_4(x_{(n-6)!!+1})$};

        \node at (5.5, 1.7) {$\vdots$};
        \node at (5.6, 1.25) {$\nu_3(x_{(n-6)!!+1})$};

        \node at (7.5, 1.7) {$\vdots$};
        \node at (7.5, 1.25) {$x_{(n-6)!!+1}$};

        \node at (5.5, 0.7) {$\vdots$};
        \node at (5.5, 0.25) {$\nu_3(x_1)$};

        \node at (6.5, 0.7) {$\vdots$};
        \node at (6.5, 0.25) {$\nu_2(x_1)$};

        \node at (7.5, 0.7) {$\vdots$};
        \node at (7.5, 0.25) {$x_1$};

    \end{scope}
\end{tikzpicture}
\caption{The staircase line for $\phi(\cF_{n-4})$ }
\label{staircase-line1}
\end{figure}

As for the other elements of $Y_1$, define
\[
\lambda_j:\phi(\cF_{n-4})\longrightarrow Y_1,\qquad
\lambda_j(\phi(z))
 =\rho_n\bigl(\nu_j(\rho_{n-2}(z))\bigr)
\quad(1\leq j\leq n-3),
\]
and write $\Psi_j=\lambda_j\phi$. Then
\[
Y_1=\bigsqcup_{j=1}^{n-3}\Psi_j(\cF_{n-4}),
\qquad
\Psi_j(\cF_{n-4})
 =\bigsqcup_{r=1}^{n-5}\Psi_j(Z_r).
\]
Moreover, $\Psi_1=\phi$.

For $1\leq j\leq n-4$, put $N_j=n-j-3$ and define the increasing
relabeling
\[
\gamma_j(a)=
\begin{cases}
a+2,&a\leq N_j,\\
a+3,&a\geq N_j+1.
\end{cases}
\]
A direct calculation from the definitions shows that $\Psi_j(z)$ is
obtained by applying $\gamma_j$ to every entry of $z$ and adjoining
the cycles $(1,2)$ and $(n-j,n)$. For $j=n-3$, $\Psi_{n-3}(z)$ is
obtained by adding $3$ to every entry of $z$ and adjoining
$(1,2)$ and $(3,n)$. These descriptions make the omitted label in
each recursive family explicit.
For $1\leq j\leq n-4$ and $1\leq r\leq n-5$, define
\[
\kappa_j(r)=
\begin{cases}
r+2,&r\leq n-j-4,\\
r+3,&r\geq n-j-3,
\end{cases}
\qquad\text{and set}\qquad
\kappa_{n-3}(r)=n-1.
\]

\begin{proposition}\label{vert-stair-prop2}
Let $1\leq j\leq n-3$, $1\leq r\leq n-5$, and
$x\in\lambda_j\phi(Z_r)$. For $3\leq k\leq n-1$,
\[
s_1\nu_k(x)s_1=s_k\nu_k(x)s_k
\quad\Longleftrightarrow\quad
k=\kappa_j(r).
\]
For $k=\kappa_j(r)$, we also have
\[
\tau_\m(\nu_{k-1}(x))
 \subset\tau_\m(\nu_k(x)),
\]
so the two vertices are not connected by a bidirected edge.
\end{proposition}
\begin{proof}
First suppose that $j\leq n-4$ and write $x=\Psi_j(z)$ with
$z\in Z_r$. Since $z(1)=r+1$, the relabeling description above shows
that $x$ contains the cycle
\[
(3,\gamma_j(r+1)).
\]
The calculation in Proposition~\ref{vert-stair-prop} therefore gives
the reflection equality exactly when
\[
k=\gamma_j(r+1)-1.
\]
If $r\leq n-j-4$, this value is $r+2$; if
$r\geq n-j-3$, it is $r+3$. This is precisely $\kappa_j(r)$.

For $j=n-3$, every $x=\Psi_{n-3}(z)$ contains $(3,n)$, so the same
calculation gives the unique critical column $k=n-1$.

In a nonterminal family, $x$ contains $(3,k+1)$ at the critical
column. Hence $s_k\in\tau_\m(x)$ and
$s_{k+1}\notin\tau_\m(x)$, and Lemma~\ref{tau-lem1} gives the strict
inclusion. In the terminal family, $x(n)=3$, so
$s_{n-1}\in\tau_\m(x)$; the terminal assertion of
Lemma~\ref{tau-lem1} gives the same conclusion.
\end{proof}

For $1\leq j\leq n-4$ and $1\leq r\leq n-6$, with
$r\neq n-j-4$, define
\[
a_j(r)=
\begin{cases}
r+3,&r\leq n-j-5,\\
r+4,&r\geq n-j-3.
\end{cases}
\]

\begin{proposition}\label{hori-stair-prop2}
Let $1\leq j\leq n-4$, $1\leq r\leq n-6$, and
$r\neq n-j-4$. Suppose that $z\in Z_r$ satisfies
$z(2)\neq r+2$, and put $x=\lambda_j\phi(z)$. Then
\[
y=s_{a_j(r)}xs_{a_j(r)}
 =\lambda_j\phi(s_{r+1}zs_{r+1})
 \in\lambda_j\phi(Z_{r+1}).
\]
The vertices $\nu_{\kappa_j(r)}(x)$ and
$\nu_{\kappa_j(r)}(y)$ are not connected by a bidirected edge.
\end{proposition}
\begin{proof}
Put $N=n-j-3$ and $p=\kappa_j(r)$. Since the missing transition
$r=N-1$ is excluded, the two labels
$\gamma_j(r+1)$ and $\gamma_j(r+2)$ are consecutive. If
$r\leq N-2$, they are $r+3,r+4$; if $r\geq N$, they are
$r+4,r+5$. Thus
\[
a_j(r)=p+1
\]
and the displayed formula for $y$ follows from the relabeling
description of $\Psi_j$. Since conjugation by $s_{r+1}$ changes the
first value of $z$ from $r+1$ to $r+2$, the element on the right lies
in $\lambda_j\phi(Z_{r+1})$.

Set $u=\nu_p(x)$ and $v=\nu_p(y)$. Lemma~\ref{nu-lem} gives
\[
v=(s_ps_{p+1}s_p)u(s_ps_{p+1}s_p).
\]
If $v=s_tus_t$ for a simple reflection $s_t$, then
Lemma~\ref{reflec-lem} says that either
$s_t=s_ps_{p+1}s_p$, which is not simple, or
\[
s_t=u(s_ps_{p+1}s_p)u
   =\sigma_px s_{p+1}x\sigma_p^{-1}.
\]
At the critical column, $x(p+1)=3$, while
$p+2=\gamma_j(r+2)$. Hence the second reflection can be simple only
if $z(r+2)=2$, equivalently $z(2)=r+2$. In the case $N=1$, the image
of $\gamma_j$ contains no label $4$, so this simple reflection cannot
occur at all. The assumed inequality therefore makes condition
\textit{BE1} fail, and Lemma~\ref{bidi-lem} proves the result.
\end{proof}

\begin{definition}\label{row-Beissinger-definition}
For $w\in\cF_n$, order its cycles as
\[
 w=(a_1,b_1)\cdots(a_r,b_r),
 \qquad a_j<b_j,\qquad b_1<\cdots<b_r.
\]
The \emph{row Beissinger insertion tableau} $P_{\rm rB}(w)$ is constructed
as follows \cite[Section~3]{Beissinger}; see also
\cite[Section~2.2]{MZ3}.  Starting with the empty tableau, at step $j$
Schensted-row-insert $a_j$: replace the leftmost entry in the first row that
is larger than $a_j$, bump that entry to the next row, and continue, appending
the current entry when no larger entry exists \cite{Schensted}.  If the new
box is created in row $t$, append $b_j$ to the end of row $t+1$.  Define the
\emph{insertion shape} by
\[
 \operatorname{sh}_{\m}(w)
   :=\operatorname{shape}\bigl(P_{\rm rB}(w)\bigr).
\]
\end{definition}

\begin{theorem}[Marberg, Zhang \cite{MZ3}, Theorems 3.15, 3.18]
\label{shape-molecule-class}
Two vertices $x,y\in\cF_n$ belong to the same molecule of $\Gamma^\m_n$ if
and only if
\[
 \operatorname{sh}_{\m}(x)=\operatorname{sh}_{\m}(y).
\]
\end{theorem}

\begin{proposition}\label{difmole-prop}
Let $x \in Y_1$ and $3\leq k\leq n-1$. If
\[
s_1\nu_k(x)s_1=s_k\nu_k(x)s_k,
\]
then $\nu_k(x)$ and $\nu_i(x)$ belong to different molecules of
$\cF_n$ for every $1\leq i<k$.
\end{proposition}

\begin{proof}
Set $u_i=\nu_i(x)$. Since $u_k\in Y_k$, Theorem~\ref{recurs-thm} gives
$u_k(1)=k+1$. The displayed equation and Lemma~\ref{reflec-lem} imply
\[
u_ks_1u_k=s_k,
\]
and hence $u_k(2)=k$. Thus the two cycles of $u_k$ containing $1$ and
$2$ are
\[
(1,k+1)\qquad\text{and}\qquad(2,k).
\]
On the other hand, $u_{k-1}=s_ku_ks_k$ has the two cycles
\[
(1,k)\qquad\text{and}\qquad(2,k+1).
\]

Apply the local rule for row Beissinger insertion to these two cycle
pairs; see~\cite[Section~2.2]{MZ3}. Replacing
$(1,k),(2,k+1)$ by $(1,k+1),(2,k)$ moves the second box produced by
these pairs to the next row of the insertion diagram. Consequently,
\[
\operatorname{sh}_{\m}(u_k)
 \neq \operatorname{sh}_{\m}(u_i)
 \qquad(1\leq i<k).
\]
For $i<k-1$, this follows by applying the same local rule successively
to $u_i,u_{i+1},\ldots,u_{k-1}$: before the final step the nested pair
$(1,k+1),(2,k)$ has not occurred, so the box moved at the final step
remains in a different row. Therefore $u_k$ and $u_i$ have different
insertion shapes and hence lie in different molecules.

This argument also covers the three exceptional horizontal phenomena
described in the following remark.  Propositions~\ref{vert-stair-prop2}
and~\ref{hori-stair-prop2} describe the corresponding generic direct
edges but are not needed to exclude a path through the exceptional
ones.
\end{proof}

\begin{remark}\label{lambda-stair-exceptions}
Proposition~\ref{hori-stair-prop2} deliberately excludes three
different phenomena. If $z(2)=r+2$, the second reflection in its
proof may be simple and the corresponding horizontal pair may be
bidirected. If $r=n-j-4$, the relabeling $\gamma_j$ separates
$r+1$ and $r+2$ by the omitted label, so the lower-rank generator
becomes a non-simple reflection and the critical column jumps by two.
Finally, for $j=n-3$ all rows have the terminal critical column
$n-1$, so there is no ordinary horizontal staircase.

Consequently, Figures~\ref{lambda staircase line}
and~\ref{lambdan-2} are schematic descriptions of the generic local
edges only. No global separation of molecules is inferred from these
figures; Proposition~\ref{difmole-prop} is proved independently by
insertion shapes.
\end{remark}

The generic segments are shown in Figure~\ref{lambda staircase line}.
For $j=n-3$, Proposition~\ref{vert-stair-prop2} gives the vertical line
depicted in Figure~\ref{lambdan-2}; no general horizontal non-bidirected
claim is made for this family.

%

    \begin{figure}
        \centering
\begin{tikzpicture}[x=1.3cm, y=1.3cm]

    \draw[gray!50, thin] (0,0) rectangle (9,5);

    \foreach \x in {1,2,...,8} {
        \draw[densely dotted, gray!70, semithick] (\x,0) -- (\x,5);
    }
    \foreach \y in {1,2,3,4} {
        \draw[densely dotted, gray!70, semithick] (0,\y) -- (9,\y);
    }

    \draw[black, thick] (2,5) -- (2,4) -- (3,4) -- (3,3) -- (4,3) -- (4,2) -- (6,2) -- (6,1) -- (7,1) -- (7,0);

    \node at (2.5, 3.5) {$\vdots$};
    \node at (3.5, 2.5) {$\cdots$};
    \node at (6.5, 1.5) {$\cdots$};

    \node[above] at (0.5, 5.1) {$Y_{n-1}$};
    \node[above] at (1.5, 5.1) {$Y_{n-2}$};
    \node[above] at (2.5, 5.1) {$Y_{n-3}$};
    \node[above] at (3.5, 5.1) {$Y_{n-4}$};
    \node[above] at (4.5, 5.1) {$\cdots$};
    \node[above] at (5.5, 5.1) {$\cdots$};
    \node[above] at (6.5, 5.1) {$Y_3$};
    \node[above] at (7.5, 5.1) {$Y_2$};
    \node[above] at (8.5, 5.1) {$Y_1$};

    \node[right] at (9.3, 4.5) {$\lambda_j\phi(Z_{n-5})$};
    \node[right] at (9.3, 3.5) {$\vdots$};
    \node[right] at (9.3, 2.5) {$\lambda_j\phi(Z_{n-j})$};
    \node[right] at (9.3, 1.5) {$\vdots$};
    \node[right] at (9.3, 0.5) {$\lambda_j\phi(Z_1)$};

\end{tikzpicture}
        \caption{The staircase line for $1\le j \neq n-3$}
        \label{lambda staircase line}
    \end{figure}
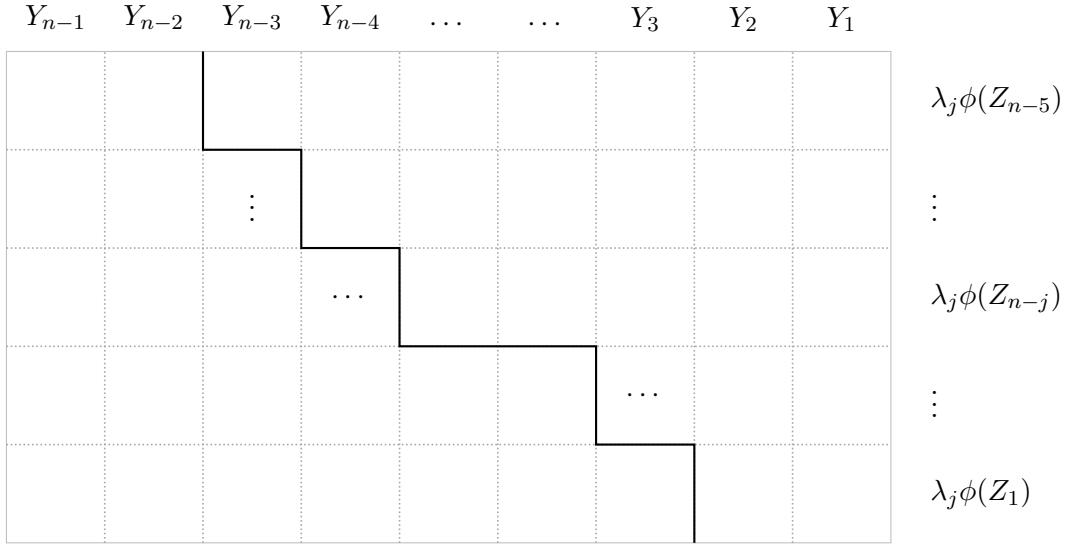

    \begin{figure}
        \centering
\begin{tikzpicture}[x=1.3cm, y=1.3cm]

    \draw[gray!50, thin] (0,0) rectangle (9,5);

    \foreach \y in {1,2,3,4} {
        \draw[densely dotted, gray!70, semithick] (0,\y) -- (9,\y);
    }

    \foreach \x in {2,3,...,8} {
        \draw[densely dotted, gray!70, semithick] (\x,0) -- (\x,5);
    }

    \draw[black, thin] (1,0) -- (1,5);

    \node[above] at (0.5, 5.1) {$Y_{n-1}$};
    \node[above] at (1.5, 5.1) {$Y_{n-2}$};
    \node[above] at (2.5, 5.1) {$Y_{n-3}$};
    \node[above] at (3.5, 5.1) {$Y_{n-4}$};
    \node[above] at (4.5, 5.1) {$\cdots$};
    \node[above] at (5.5, 5.1) {$\cdots$};
    \node[above] at (6.5, 5.1) {$Y_3$};
    \node[above] at (7.5, 5.1) {$Y_2$};
    \node[above] at (8.5, 5.1) {$Y_1$};

    \node[right] at (9.3, 4.5) {$\lambda_{n-3}\phi(Z_{n-5})$};
    \node[right] at (9.3, 3.5) {$\vdots$};
    \node[right] at (9.3, 2.5) {$\lambda_{n-3}\phi(Z_{n-j})$};
    \node[right] at (9.3, 1.5) {$\vdots$};
    \node[right] at (9.3, 0.5) {$\lambda_{n-3}\phi(Z_1)$};

\end{tikzpicture}
        \caption{The staircase line for $j=n-3$}
        \label{lambdan-2}
    \end{figure}
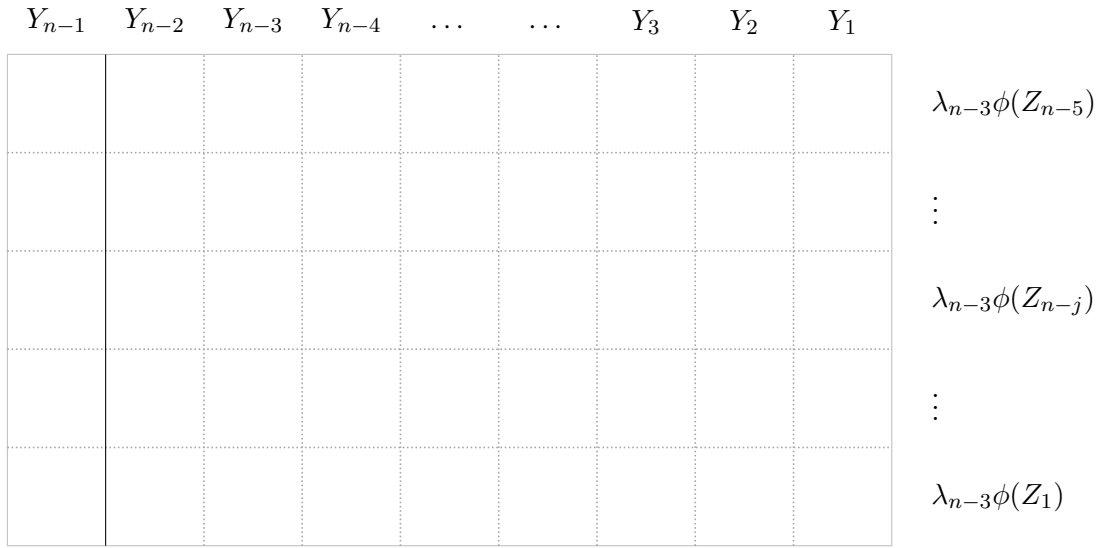


\section{Molecules}\label{mole-sect}

We can show the molecules in Figure~\ref{figure1}, such that each grid is a element of $\cF_n$.  After the elements in a molecule of $\cF_{n-2}$ are mapped to $Y_1$ by $\rho$, they are still contained in the same molecule. These elements are then connected to $Y_2, Y_3, \cdots, Y_{n-1}$ as much as possible, expanding the molecule. As mentioned in the previous section, the staircase line provides boundaries of molecules.
	\begin{figure}
		\begin{center}
			\includegraphics[scale=0.45]{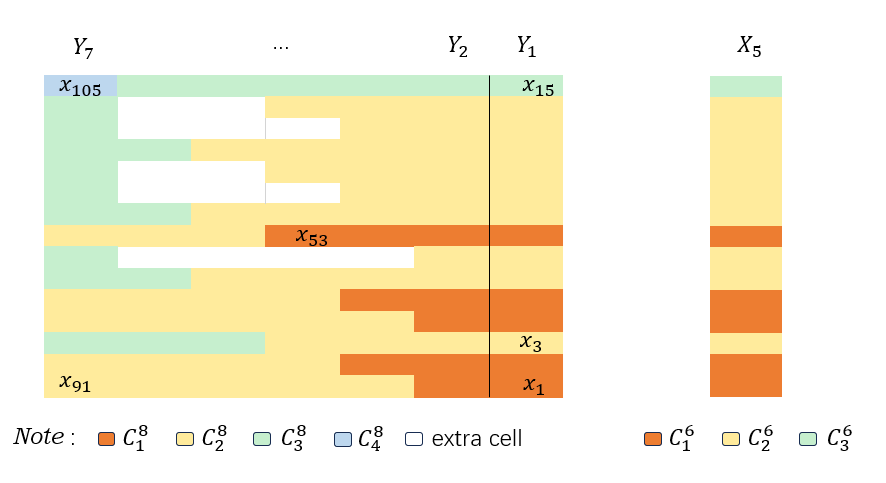}
		\end{center}
		\caption{The example for $ n=8 $ and $n=6$ }
        \label{figure1}
	\end{figure}

\begin{corollary}\label{merge-cor}
Assume $x,y \in \cF_{n-2}$ are in different molecules, then $\rho(x)$ and $\rho(y)$ are in different molecules in $\cF_n$.
\end{corollary}
\begin{proof}
Suppose, to the contrary, that $\rho(x)\sim\rho(y)$. The final
assertion of Proposition~\ref{rho-bi-prop} then implies $x\sim y$,
contradicting the hypothesis. Hence $\rho(x)$ and $\rho(y)$ belong to
different molecules in $\cF_n$.
\end{proof}

Corollary \ref{merge-cor} also gives a useful lower bound for the number of
molecules.

\begin{proposition}
Let $m_n$ be the number of molecules of $\Gamma^\m_n$.  Then
\[
                  m_n\geq m_{n-2}+1.
\]
Consequently, if $n=2p$, then $m_n\geq p$.
\end{proposition}
\begin{proof}
By Proposition \ref{rho-bi-prop}, the image under $\rho$ of every molecule
of $\Gamma^\m_{n-2}$ is contained in a molecule of
$\Gamma^\m_n$.  Corollary \ref{merge-cor} says that the images of
two distinct molecules cannot lie in the same molecule.  On the other
hand, $\tau_\m(w_0)=S$, so no bidirected edge is incident with $w_0$.
Thus $\{w_0\}$ is a molecule, and it is not one of the molecules meeting
$Y_1=\rho(\cF_{n-2})$.  This proves the recurrence.  Since $m_2=1$,
induction gives $m_{2p}\geq p$.
\end{proof}

For $1\leq k\leq p$, write
\[
             C_k^n=\mathscr C\bigl(x_{(2k-1)!!}\bigr).
\]
In particular, $C_p^n=\{w_0\}$.
To clarify the recursive structure of $C_1^n$, we have the following theorem.
\begin{theorem}\label{C1-thm1}
\[
                  C_1^n\cap Y_1=\rho(C_1^{n-2}).
\]
\end{theorem}
\begin{proof}
The labeling of the distinguished vertices gives
$x_1^n=\rho(x_1^{n-2})$.  If $u\in C_1^{n-2}$, Proposition
\ref{rho-bi-prop} shows that $\rho(u)$ and
$\rho(x_1^{n-2})=x_1^n$ belong to the same molecule.  Hence
$\rho(C_1^{n-2})\subseteq C_1^n\cap Y_1$.

Conversely, let $z\in C_1^n\cap Y_1$.  There is a unique
$u\in \cF_{n-2}$ such that $z=\rho(u)$.  If
$u\notin C_1^{n-2}$, then $u$ and $x_1^{n-2}$ lie in different
molecules of $\cF_{n-2}$.  Corollary \ref{merge-cor} would then imply
that $\rho(u)=z$ and $\rho(x_1^{n-2})=x_1^n$ lie in different
molecules of $\cF_n$, contrary to $z\in C_1^n$.  Thus
$u\in C_1^{n-2}$ and $z\in\rho(C_1^{n-2})$.
\end{proof}

\begin{definition}
For $n=2p$, an element $w\in \cF_n$ is called a \emph{middle element} if
\[
                         \tau_\m(w)=\{s_p\}.
\]
\end{definition}

\begin{definition}\label{sum-permutation-definition}
For $\alpha\in S_a$ and $\beta\in S_b$, their \emph{direct sum} and
\emph{skew sum} are the permutations in $S_{a+b}$ defined by
\[
 (\alpha\oplus\beta)(i)=
 \begin{cases}
  \alpha(i),&1\leq i\leq a,\\
  a+\beta(i-a),&a<i\leq a+b,
 \end{cases}
\]
and
\[
 (\alpha\ominus\beta)(i)=
 \begin{cases}
  b+\alpha(i),&1\leq i\leq a,\\
  \beta(i-a),&a<i\leq a+b.
 \end{cases}
\]
Thus the first block of $\alpha\oplus\beta$ lies southwest of the second
block in the permutation diagram, while the first block of
$\alpha\ominus\beta$ lies northwest of the second.
\end{definition}

\begin{theorem}[Greene \cite{Greene}]\label{Greene-theorem}
If $\lambda=\operatorname{sh}(w)$, then for every $k\geq1$,
\[
 \lambda_1+\cdots+\lambda_k
 =
 \max\bigl\{
 |I_1\cup\cdots\cup I_k|:
 I_1,\ldots,I_k\text{ are increasing subsequences of }w
 \bigr\}.
 \tag{6.1}\label{Greene-invariant}
\]
\end{theorem}

\begin{lemma}\label{shape-transport-lem}
Let $x\in\cF_{n-2}$ and
$\operatorname{sh}(x)=(\lambda_1,\ldots,\lambda_\ell)$.  Then
\[
 \operatorname{sh}(\rho(x))
   =(\lambda_1+1,\lambda_2+1,\lambda_3,\ldots,\lambda_\ell)
\]
and
\[
 \operatorname{sh}(\nu_{n-1}\rho(x))
   =(\lambda_1,\ldots,\lambda_\ell,1,1).
\]
\end{lemma}
\begin{proof}
The definitions of $\rho$ and $\nu_{n-1}$ give
\[
 \rho(x)=21\mathbin{\oplus}x,
 \qquad
 \nu_{n-1}\rho(x)=1\mathbin{\ominus}x\mathbin{\ominus}1.
\]
By Theorem~\ref{Greene-theorem}, the Greene invariants of a direct sum
are the sums of the invariants of its two blocks, whereas the row lengths
for a skew sum are the multiset union of the row lengths of its blocks.
Since $\operatorname{sh}(21)=(1,1)$ and
$\operatorname{sh}(1)=(1)$, the two displayed formulas follow.
\end{proof}

\begin{proposition}\label{mid-uniq-prop}
There is a unique middle element in $\cF_n$, namely
\[
 w_n=(1,p+1)(2,p+2)\cdots(p,2p).
\]
Moreover,
\[
                         w_n=\nu_p\rho(w_{n-2}).
\]
\end{proposition}
\begin{proof}
For an FPF involution $w$, the definition of $\tau_\m$ gives
\[
 \tau_\m(w)=\{s_i:w(i)>w(i+1)\}.
\]
If $\tau_\m(w)=\{s_p\}$, then $w$ is increasing on each of the intervals
$[1,p]$ and $[p+1,2p]$.  No transposition $(a,b)$ of $w$ can have both
endpoints in one of these intervals: if $a<b$ were in the same interval,
then $w(a)=b>a=w(b)$ would contradict the required increase.  Hence every
edge joins the two intervals.  The increasing sequence
$w(1),\ldots,w(p)$ must therefore be $p+1,\ldots,2p$, so
$w(i)=p+i$ and $w=w_n$.  Conversely this $w_n$ has its unique descent at
$p$, and hence is a middle element.  Finally, direct evaluation of $\rho$
and $\nu_p$ on the displayed transpositions gives
$\nu_p\rho(w_{n-2})=w_n$.
\end{proof}

\begin{definition}\label{nonnesting-definition}
View an element of $\cF_n$ as a perfect matching on $[n]$.  It is called
\emph{nonnesting} if it has no two arcs $(a,d)$ and $(b,c)$ with
$a<b<c<d$.
\end{definition}

\begin{lemma}\label{nonnesting-shape}
For $n=2p$ and $x\in\cF_n$,
\[
 x\text{ is nonnesting}
 \quad\Longleftrightarrow\quad
 \operatorname{sh}(x)=(p,p).
\]
\end{lemma}
\begin{proof}
If $(a,d)$ and $(b,c)$ form a nesting with $a<b<c<d$, then the entries
of the one-line notation in positions $a,b,c,d$ are $d,c,b,a$, a
decreasing subsequence of length four.  Conversely, in a nonnesting
matching the values at the left endpoints increase from left to right,
and the same holds at the right endpoints.  A descent can only go from
a left endpoint to a right endpoint, so there is no decreasing
subsequence of length three.

Schensted's theorem states that the length of a longest decreasing
subsequence of $w$ is the length of the first column of
$\operatorname{sh}(w)$ \cite{Schensted}.  Moreover, for an involution
the number of fixed points equals the number of odd-length columns in
its Robinson--Schensted tableau \cite[Section~3]{Beissinger}.  Thus an
FPF involution has only even column lengths.  For a nonnesting element
the first column has length at most two, so every nonzero column has
length two; since there are $2p$ boxes, its shape is $(p,p)$.  In the
other direction, a nesting gives a decreasing subsequence of length
four, so Schensted's theorem rules out the shape $(p,p)$.
\end{proof}

\begin{theorem}\label{D1-rec-thm}
For every even $n$,
\[
                         D_1^n=C_1^n.
\]
\end{theorem}
\begin{proof}
We prove by induction that $D_1^n$ is exactly the set of
nonnesting matchings in $\cF_n$.  This is clear for $n=2$.  Every
$x\in \cF_n$ has a unique expression
\[
                         x=\nu_j(z),
 \qquad z=\rho(u),\quad u\in \cF_{n-2},\quad 1\leq j\leq n-1.
\]
The pair $(1,2)$ of $z$ becomes $(1,j+1)$ in $x$, and deleting this
pair and standardizing the remaining labels recovers $u$.  Put
$k_z=z(3)-1$.  If $u$ is nonnesting, then the first residual arc of
$z$ starts at $3$ and ends at $z(3)$.  After applying $\nu_j$, the
new arc $(1,j+1)$ contains a residual arc precisely when
\[
                         z(3)\leq j+1,
\]
or equivalently when $j\geq k_z$.  If $j<k_z$, that first residual
arc crosses the right endpoint $j+1$; any residual arc lying entirely
inside $(1,j+1)$ would be nested inside it, which is impossible.
Consequently,
\[
 \nu_j(\rho(u))\text{ is nonnesting}
 \quad\Longleftrightarrow\quad
 u\text{ is nonnesting and }1\leq j<k_z.
\]
This is exactly the recursion defining $D_1^n$, so the induction is
complete.

Finally, $x_1$ is nonnesting.  By Lemma~\ref{nonnesting-shape} and
Theorem~\ref{shape-molecule-class}, its molecule $C_1^n$ is precisely the
set of nonnesting matchings.  Therefore $D_1^n=C_1^n$.
\end{proof}

When $n=4$ and $n=6$, the molecule $C_1^n$ is shown in Figure \ref{figure1}. We can see how the recursive structure works immediately. The molecule $C_1^n$ is determined by the previously identified molecule $C_1^{n-2}$, following a straightforward rule. 
        
For $1\leq k\leq p$, set
\[
 \lambda_{p,k}
   =(p-k+1,p-k+1,1^{\,2k-2}).
\]

\begin{lemma}\label{distinguished-shape-lem}
For $n=2p$ and $1\leq k\leq p$,
\[
       \operatorname{sh}\bigl(x_{(2k-1)!!}\bigr)
       =\lambda_{p,k}.
\]
\end{lemma}
\begin{proof}
At rank $2k$, the labeling gives
$x_{(2k-1)!!}=w_0$, whose one-line notation is decreasing and whose
shape is $(1^{2k})$.  For every subsequent increase of the rank by
two, the same labeled vertex is carried forward by $\rho$.  Applying
Lemma \ref{shape-transport-lem} exactly $p-k$ times changes
$(1^{2k})$ into
\[
                    (p-k+1,p-k+1,1^{2k-2}),
\]
as required.
\end{proof}

\begin{theorem}\label{main thm}
If $1<k<p$, then
\[
\begin{split}
 C_k^n
 &=\mathscr C\bigl(x_{(2k-1)!!}\bigr)\\
 &=\mathscr C\bigl(\nu_{n-1}\rho(C_{k-1}^{n-2})\bigr)\\
 &=\mathscr C\bigl(\rho(C_k^{n-2})\bigr).
\end{split}
\]
\end{theorem}
\begin{proof}
By Lemma \ref{distinguished-shape-lem}, $C_k^n$ is the molecule of
shape
\[
                         \lambda_{p,k}
       =(p-k+1,p-k+1,1^{2k-2}).
\]
Every element of $C_k^{n-2}$ has shape
\[
                         \lambda_{p-1,k}
       =(p-k,p-k,1^{2k-2}).
\]
The first formula in Lemma \ref{shape-transport-lem} therefore shows
that every element of $\rho(C_k^{n-2})$ has shape
$\lambda_{p,k}$.

Likewise, every element of $C_{k-1}^{n-2}$ has shape
\[
                         \lambda_{p-1,k-1}
       =(p-k+1,p-k+1,1^{2k-4}).
\]
The second formula in Lemma \ref{shape-transport-lem} shows that every
element of $\nu_{n-1}\rho(C_{k-1}^{n-2})$ also has shape
$\lambda_{p,k}$.  Both displayed image sets are nonempty, and by
Theorem~\ref{shape-molecule-class} the full fiber of the shape
$\lambda_{p,k}$ is the single molecule $C_k^n$.  Taking the molecule
containing either image set gives the asserted equalities.
\end{proof}
We illustrate these recursive descriptions in Figure \ref{figure1}. The elements of the same molecule are represented by the same color. For cyan, blue and orange molecules, we can find their specific elements. The mapping $\nu_8$ permutes the molecules of $\cF_6$. However, there are not only $p$ molecules in $\cF_n$, but also some additional molecules, such as the white part in Figure \ref{figure1} and it generates molecule in larger $\cF_n$.

\section{Cells}\label{cell-sect}

The insertion-shape classification in
Theorem~\ref{shape-molecule-class} classifies molecules, not cells.  We do
not assume any general classification of the cells of the Gelfand
$S_n$-graphs.  Since every bidirected edge is directed in both directions,
each molecule is contained in a unique cell.  Let $\mathcal K_1^n$ denote
the cell containing $C_1^n$.  The question in this section is whether
$\mathcal K_1^n=C_1^n$.

We now prove that the distinguished molecule is a cell.  The argument does
not use a classification of cells.  Instead, it constructs a quotient onto a
two-row Specht module and identifies its canonical-basis kernel by integral
Garnir straightening.

\subsection{A two-row Specht quotient}

For a partition $\lambda\vdash n$, let $\operatorname{Std}(\lambda)$ denote
the set of standard Young tableaux of shape $\lambda$: each of
$1,\ldots,n$ occurs exactly once, and the entries increase from left to right
in every row and from top to bottom in every column.  A tableau is
\emph{column-standard} if its entries increase from top to bottom in each
column; no condition is imposed on its rows.

Fix $n=2r$, and write
\[
 x=(a_1,b_1)\cdots(a_r,b_r),\qquad
 a_1<\cdots<a_r,\quad a_j<b_j .
\]
Associate to $x$ the column-standard tableau
\[
 T_x=\begin{pmatrix}
 a_1&\cdots&a_r\\
 b_1&\cdots&b_r
 \end{pmatrix}
\]
of shape $(r,r)$.  The tableau $T_x$ is standard if and only if
$b_1<\cdots<b_r$, equivalently, if and only if $x$ is nonnesting.  Thus
\[
 x\in C_1^n\quad\Longleftrightarrow\quad
 T_x\in\operatorname{Std}(r,r).
 \tag{7.1}\label{C1-standard}
\]

We use the following concrete meaning of the two-row Specht module.  Over
$\mathcal A=\mathbb Z[v,v^{-1}]$, the module $S^{(r,r)}$ is the
$\mathcal H_n$-module obtained from the free module on tableaux of shape
$(r,r)$ by imposing Welsh's Hecke column relations and Garnir relations
\cite[(4) and (7)]{Welsh}.  Adjacent labels are first permuted and the result
is then straightened with these relations.  Welsh's standard-basis result \cite[Section~2]{Welsh} says that
\[
 \{\mathbf v_T:T\in\operatorname{Std}(r,r)\}
\]
is an $\mathcal A$-basis of $S^{(r,r)}$; for a nonstandard $T$,
$\mathbf v_T$ denotes its class and must be straightened in this basis.
Thus ``two-row'' refers to the indexing partition $(r,r)$, rather than to a
new module attached to the matching model.

Let
\[
 t^-=\begin{pmatrix}1&3&\cdots&2r-1\\2&4&\cdots&2r\end{pmatrix}
\]
be the column-superstandard tableau.  The \emph{column-reading word} of a
tableau is obtained by reading each column from top to bottom and the columns
from left to right.  For a tableau $T$, let $w_T$ be the unique permutation
which sends $t^-$ to $T$ by permuting entries in these positions.  We put
\[
 e_T=v^{-\ell(w_T)}\mathbf v_T.
 \tag{7.2}\label{normalized-tableau}
\]
Welsh uses generators $h_i$ and a parameter $q$ satisfying
$h_i^2=(q-1)h_i+q$.  Throughout this section we make the substitution
\[
 q=v^2,\qquad h_i=vH_i.
 \tag{7.3}\label{parameter-conversion}
\]
It converts Welsh's action into the Hecke algebra convention of this paper.

An elementary inversion count gives
\[
 \ell(w_{T_x})=\h(x).
 \tag{7.4}\label{height-tableau-length}
\]
Indeed, for each pair of edges, the separated, crossing, and nested endpoint
orders contribute respectively $0$, $1$, and $2$ inversions to the
column-reading word, exactly as in the definition of $\h$.

We also fix the convention needed to apply the web-basis results.  For
$u,d\in S_{2r}$, write $u\leq_{\rm wk}d$ in right weak order if a reduced
expression for $u$ is an initial segment of a reduced expression for $d$.
Let $t^{(2^r)}=(t^-)^{\mathsf t}$ and set
\[
 \mathcal D^{\rm st}_{(2^r)}
 =\{d\in S_{2r}:t^{(2^r)}d\in\operatorname{Std}((2^r))\},
 \qquad \eta(d)=t^{(2^r)}d.
\]
Thus $\eta$ transports right weak order to
$\operatorname{Std}((2^r))$.  Heard and Kujawa use right modules.  If
$z_{(r,r)}$ denotes their cyclic Specht generator, their standard-basis
formula is
\[
 v_t^{\rm HK}:=z_{(r,r)}T_{\eta^{-1}(t)},
 \qquad t\in\operatorname{Std}((2^r)).
\]
Here $T_w$ is the standard Hecke element attached to a reduced expression of
$w$.  In particular, if
$\eta^{-1}(t')=\eta^{-1}(t)s_i$ and the length increases by one, then
$v_t^{\rm HK}T_i=v_{t'}^{\rm HK}$.  This is the cover formula used below
\cite[Section~2.5]{HK}.

Transporting their right action through the anti-involution $H_i\mapsto H_i$
and transposing tableaux identifies
\[
 e_T\longleftrightarrow v_{T^{\mathsf t}}^{\rm HK}
 \qquad(T\in\operatorname{Std}(r,r)),
 \tag{7.5}\label{HK-basis-identification}
\]
after setting their parameter equal to $v$.  Both sides agree on $t^-$, and
the cover formula just stated agrees with Welsh's normalized generator
formula.  Induction on $\ell(w_T)$ proves the identification.  Thus every
use below of the standard-to-web transition matrix applies to the normalized
basis $\{e_T\}$, not to an unnormalized tableau basis.

We next derive the local straightening identities.  Fix $a<b<c<d$, keep all
other columns fixed, and abbreviate the following adjacent two-column
configurations by
\[
 A=\begin{pmatrix}a&c\\ b&d\end{pmatrix},\qquad
 X=\begin{pmatrix}a&b\\ c&d\end{pmatrix},\qquad
 Z=\begin{pmatrix}a&b\\ d&c\end{pmatrix}.
\]
Let $A^{\rm rev},X^{\rm rev},Z^{\rm rev}$ denote the tableaux obtained by
reversing the order of these two columns.  Explicitly,
\[
 A^{\rm rev}=\begin{pmatrix}c&a\\d&b\end{pmatrix},\quad
 X^{\rm rev}=\begin{pmatrix}b&a\\d&c\end{pmatrix},\quad
 Z^{\rm rev}=\begin{pmatrix}b&a\\c&d\end{pmatrix}.
\]

\begin{lemma}\label{two-column-straightening}
In the normalization \eqref{normalized-tableau}, one has
\[
 e_Z=ve_X-v^2e_A,
 \tag{7.6}\label{Z-straightening}
\]
and
\[
 e_{A^{\rm rev}}=v^2e_A,\qquad
 e_{X^{\rm rev}}=v(e_A+e_Z),\qquad
 e_{Z^{\rm rev}}=e_Z.
 \tag{7.7}\label{reversed-straightening}
\]
These identities remain valid inside a larger tableau when the displayed
columns are adjacent and all complementary columns are fixed.
\end{lemma}

\begin{proof}
Write $\ell_0$ for the contribution to $\ell(w_T)$ from the complementary
columns and from their comparisons with the displayed entries.  Since the
same four entries occupy the same four adjacent positions, this contribution
is common to all six tableaux.  Their internal inversion numbers are
\[
\begin{array}{c|cccccc}
 U&A&X&Z&A^{\rm rev}&X^{\rm rev}&Z^{\rm rev}\\ \hline
 \ell(w_U)-\ell_0&0&1&2&4&3&2.
\end{array}
 \tag{7.8}\label{local-length-table}
\]
Welsh's three-entry Garnir relation \cite[(7)]{Welsh}, followed where
necessary by the column relation \cite[(4)]{Welsh}, gives the following
relations for the unnormalized vectors:
\[
\begin{aligned}
 \mathbf v_Z-q\mathbf v_X+q^2\mathbf v_A&=0,\\
 \mathbf v_{A^{\rm rev}}+q\mathbf v_Z-q^2\mathbf v_X&=0,\\
 \mathbf v_{X^{\rm rev}}-q\mathbf v_Z-q^2\mathbf v_A&=0,\\
 \mathbf v_{Z^{\rm rev}}-q\mathbf v_X+q^2\mathbf v_A&=0.
\end{aligned}
 \tag{7.9}\label{raw-local-relations}
\]
For completeness, the second, third, and fourth lines are obtained by applying
the Garnir relation to the top-row descents of
$A^{\rm rev},X^{\rm rev},Z^{\rm rev}$, respectively.  A term with a descending
right column is replaced by its negative after reversing that column.  No
other entries move, which is why the calculation is local.

The first line of \eqref{raw-local-relations} gives
$\mathbf v_Z=q\mathbf v_X-q^2\mathbf v_A$.  Substitution in the remaining
three lines yields
\[
 \mathbf v_{A^{\rm rev}}=q^3\mathbf v_A,\qquad
 \mathbf v_{X^{\rm rev}}=q\mathbf v_Z+q^2\mathbf v_A,\qquad
 \mathbf v_{Z^{\rm rev}}=\mathbf v_Z.
\]
Now use $q=v^2$, the length table \eqref{local-length-table}, and
\eqref{normalized-tableau}.  The resulting four equations are exactly
\eqref{Z-straightening} and \eqref{reversed-straightening}.
\end{proof}

Set
\[
 L^+_{\rm st}=
 \bigoplus_{T\in\operatorname{Std}(r,r)}\mathbb Z[v]e_T
 \subset S^{(r,r)}.
\]

\begin{lemma}\label{integral-straightening}
For every column-standard tableau $T$ of shape $(r,r)$,
\[
 e_T\in L^+_{\rm st}.
\]
If the top row of $T$ is increasing and the bottom row is not, then
\[
 e_T\in vL^+_{\rm st}.
 \tag{7.10}\label{strict-straightening}
\]
In particular, $e_{T_x}\in vL^+_{\rm st}$ whenever $x\notin C_1^n$.
\end{lemma}

\begin{proof}
If the top row has an adjacent descent, the corresponding pair of columns is
one of $A^{\rm rev},X^{\rm rev},Z^{\rm rev}$, and
Lemma~\ref{two-column-straightening} applies through
\eqref{reversed-straightening}.  If the top row is increasing but the bottom
row has an adjacent descent, the corresponding pair is $Z$, and we use
\eqref{Z-straightening}.  Each resulting tableau is column-standard.

Order tableaux lexicographically by
\[
 \bigl(\ell(w_T),\operatorname{inv}(\text{top row of }T)\bigr).
\]
The $A^{\rm rev}$ and $X^{\rm rev}$ moves lower the first coordinate, the
$Z^{\rm rev}$ move preserves the first coordinate and lowers the second, and
the $Z$ move lowers the first coordinate.  Comparisons with entries outside
the displayed adjacent columns are unchanged.  The process therefore
terminates, and all its coefficients lie in $\mathbb Z[v]$.

Under the additional hypothesis in the second assertion, the first move is
necessarily \eqref{Z-straightening}; both of its coefficients are divisible
by $v$.  Every later coefficient belongs to $\mathbb Z[v]$, so divisibility
by $v$ persists.  The final assertion follows from \eqref{C1-standard}.
\end{proof}

Put
\[
 x_0=(1,2)(3,4)\cdots(n-1,n).
\]
Let
\[
 \widehat{\mathcal M}=\bigoplus_{x\in\cF_n}\mathcal A \widehat{M}_x
\]
be the positive companion module, with left $\mathcal H_n$-action
\[
 H_i\widehat{M}_x=
 \begin{cases}
 \widehat{M}_y,&\h(y)>\h(x),\\
 \widehat{M}_y+(v-v^{-1})\widehat{M}_x,&\h(y)<\h(x),\\
 -v^{-1}\widehat{M}_x,&y=x,
 \end{cases}
 \qquad y=s_ixs_i.
 \tag{7.11}\label{Mhat-action}
\]
This is the companion of the model introduced in
Section~\ref{prelim-sect}.  Its bar involution is the semilinear map
characterized by $\overline v=v^{-1}$,
$\overline{\widehat{M}_{x_0}}=\widehat{M}_{x_0}$, and
$\overline{H\eta}=\overline H\,\overline\eta$.  Its existence is part of the
quasiparabolic companion-module construction
\cite[Theorem~3.16 and Corollary~3.20]{Marberg}; the adjective ``positive''
refers to the corresponding $v\mathbb Z[v]$ triangular convention.

\begin{proposition}\label{specht-quotient}
The map
\[
 \pi:\widehat{\mathcal M}\longrightarrow S^{(r,r)},\qquad
 \pi(\widehat{M}_x)=e_{T_x},
 \tag{7.12}\label{specht-quotient-map}
\]
is a surjective $\mathcal H_n$-module homomorphism.  Moreover,
\[
 \pi(\widehat{M}_x)\in L^+_{\rm st}\quad(x\in\cF_n),\qquad
 \pi(\widehat{M}_x)\in vL^+_{\rm st}\quad(x\notin C_1^n).
 \tag{7.13}\label{standard-image-lattice}
\]
\end{proposition}

\begin{proof}
Put $y=s_ixs_i$.  We first assume $\h(y)>\h(x)$.  Since
\eqref{height-tableau-length} identifies height with column-reading length,
Welsh's action formula and \eqref{parameter-conversion} give
\[
 H_ie_{T_x}=e_{s_iT_x},
 \tag{7.14}\label{tableau-generator-action}
\]
where $s_iT_x$ means that the entries $i$ and $i+1$ are interchanged.

We check that the right side is $e_{T_y}$.  If $i$ and $i+1$ belong to
different edges, interchanging them cannot reverse either edge, since they
are consecutive.  The canonical order of the columns can change only when
both $i$ and $i+1$ are the smaller endpoints of their edges.  In that case
write the two edges as
\[
 (i,a),\ (i+1,b),\qquad a,b>i+1.
\]
The height-increasing case is $a<b$.  Relative to the four ordered endpoints,
$T_x$ has configuration $X$, the swapped tableau $s_iT_x$ has configuration
$Z^{\rm rev}$, and $T_y$ has configuration $Z$.  Hence
$e_{s_iT_x}=e_{T_y}$ by the last identity in
\eqref{reversed-straightening}.  This exhausts the height-increasing cases.

If $\h(y)<\h(x)$, apply the already proved increasing case to $y$ and use
$H_i^2=1+(v-v^{-1})H_i$; this gives
\[
 H_ie_{T_x}=e_{T_y}+(v-v^{-1})e_{T_x}.
\]
Finally, $y=x$ occurs precisely when $(i,i+1)$ is an edge of $x$.  The two
entries then form a column, and Welsh's column relation gives
$H_ie_{T_x}=-v^{-1}e_{T_x}$.  Thus \eqref{specht-quotient-map} intertwines
all generators.

For $c\in C_1^n$, the tableaux $T_c$ run through
$\operatorname{Std}(r,r)$, so their vectors form the standard Specht basis.
Therefore $\pi$ is surjective.  The two lattice assertions are
Lemma~\ref{integral-straightening}.
\end{proof}

\subsection{The canonical-basis kernel}

Let $\operatorname{NC}_r$ be the set of noncrossing perfect matchings of
$[2r]$.  Thus every $D\in\operatorname{NC}_r$ consists of $r$ disjoint pairs
covering $[2r]$, and it has no two pairs $(a,c)$ and $(b,d)$ with
$a<b<c<d$.  For $T\in\operatorname{Std}(r,r)$, define
$\psi(T)\in\operatorname{NC}_r$ by reading the entries of the bottom row from
left to right and joining each one to the largest still-unmatched top-row
entry smaller than it.  The standardness of $T$ guarantees that this
procedure is defined and gives a noncrossing matching.  Transport right weak
order to webs by declaring
\[
 D\preceq D'
 \quad\Longleftrightarrow\quad
 T^{\mathsf t}\leq_{\rm wk}(T')^{\mathsf t}
 \quad\text{when }D=\psi(T)\text{ and }D'=\psi(T').
\]

In the diagrammatic realization of Heard and Kujawa, let $w_D$ be the vector
represented by $D\in\operatorname{NC}_r$.  Section~3.4 of \cite{HK} defines
the Hecke action on the free module with these diagram vectors as basis, and
\cite[Proposition~4.2.1]{HK} identifies that web module with $S^{(r,r)}$.
Thus
\[
 \mathcal W_r:=\{w_D:D\in\operatorname{NC}_r\}
\]
is an $\mathcal A$-basis of $S^{(r,r)}$.
Under \eqref{HK-basis-identification}, their unitriangularity and positivity
theorems give the more explicit standard-to-web formula
\[
 e_T=w_{\psi(T)}+
 \sum_{D\prec\psi(T)}a_{D,T}(v)w_D,
 \qquad a_{D,T}(v)\in\mathbb Z_{\geq0}[v],
\]
where $\prec$ is the strict part of $\preceq$ defined above
\cite[Theorems~5.1.1 and~5.2.1]{HK}.  Consequently
\[
 L^+_{\rm st}
 =\bigoplus_{w\in\mathcal W_r}\mathbb Z[v]w.
 \tag{7.15}\label{web-lattice}
\]

Let
\[
 D_0=\{(1,2),(3,4),\ldots,(2r-1,2r)\},
 \qquad w_0=w_{D_0}.
\]
This is the adjacent-cap web and $D_0=\psi(t^-)$.  In the right-module
convention, the bar involution on $S^{(r,r)}$ is the semilinear involution
normalized by $\overline{w_0}=w_0$ and
$\overline{\xi H}=\overline\xi\,\overline H$; we use its transport to the
left-module convention through $H_i\mapsto H_i$.

\begin{lemma}\label{web-bar-lemma}
Every web $w\in\mathcal W_r$ is fixed by the bar involution on
$S^{(r,r)}$.  Consequently,
\[
 \{\xi\in vL^+_{\rm st}:\overline{\xi}=\xi\}=0.
 \tag{7.16}\label{bar-lattice-intersection}
\]
\end{lemma}

\begin{proof}
Use the right-module convention temporarily, and write $T_i$ for the
generators used by Heard and Kujawa.  The bar involution satisfies
\[
 \overline{T_i}=T_i^{-1}=T_i-(v-v^{-1}).
\]
Hence
\[
 E_i:=T_i-v
 \qquad\text{satisfies}\qquad
 \overline{E_i}=E_i.
 \tag{7.17}\label{bar-invariant-web-generator}
\]
The minimal standard vector corresponds to the adjacent-cap web $w_0$ and is
bar invariant.  The web poset is identified with the weak order on standard
tableaux, and its covers are obtained by the nonzero $E_i$-moves
\cite[Proposition~3.5.1]{HK}.  Since this poset has the unique minimum $w_0$,
every web has the form
\[
 w=w_0E_{i_1}\cdots E_{i_k}
\]
for a suitable sequence of covers.  Equation
\eqref{bar-invariant-web-generator} therefore gives $\overline w=w$.
Transporting back through $H_i\mapsto H_i$ proves the left-module assertion.

Now write $\xi=\sum_w f_w(v)w\in vL^+_{\rm st}$.  Bar invariance gives
$f_w(v)=f_w(v^{-1})$ for every $w$, while
$f_w(v)\in v\mathbb Z[v]$.  Since
\[
 v\mathbb Z[v]\cap v^{-1}\mathbb Z[v^{-1}]=0,
\]
all $f_w$ vanish.
\end{proof}

By \cite[Corollary~3.20]{Marberg}, there is a unique basis
$\{\widehat{M}'_x:x\in\cF_n\}$ satisfying
\[
 \overline{\widehat{M}'_x}=\widehat{M}'_x,\qquad
 \widehat{M}'_x\in \widehat{M}_x+\sum_{z<x}v\mathbb Z[v]\widehat{M}_z.
 \tag{7.18}\label{positive-canonical-triangularity}
\]
We call it the \emph{positive canonical basis} of $\widehat{\mathcal M}$.

\begin{lemma}\label{bar-compatible-quotient}
The quotient map $\pi$ is compatible with bar involutions:
\[
 \pi(\overline\eta)=\overline{\pi(\eta)}
 \qquad(\eta\in\widehat{\mathcal M}).
 \tag{7.19}\label{bar-compatible-map}
\]
\end{lemma}

\begin{proof}
Recall that $x_0$ is the unique minimal element of $\cF_n$.  The module
$\widehat{\mathcal M}$ is
generated by $\widehat{M}_{x_0}$.  Indeed, if $x\ne x_0$, then $x$ is not
$W$-minimal, so some $s_i$ satisfies $s_ixs_i<x$.  Induction on height and
the first line of \eqref{Mhat-action}, applied in the reverse direction, shows
that $\widehat{M}_x\in\mathcal H_n\widehat{M}_{x_0}$.

The bar on $\widehat{\mathcal M}$ fixes $\widehat{M}_{x_0}$.  Under
\eqref{HK-basis-identification}, $\pi(\widehat{M}_{x_0})=e_{T_{x_0}}$ is the
adjacent-cap web $w_0$, which is bar invariant by
Lemma~\ref{web-bar-lemma}.  The two semilinear maps
\[
 \eta\longmapsto\pi(\overline\eta),
 \qquad
 \eta\longmapsto\overline{\pi(\eta)}
\]
therefore agree on $\widehat{M}_{x_0}$, and both send $H\eta$ to
$\overline H$ times the image of $\eta$.  Since $\widehat{M}_{x_0}$ generates
$\widehat{\mathcal M}$, the maps agree everywhere.
\end{proof}

\begin{lemma}\label{canonical-kernel}
One has
\[
 \ker\pi=
 \bigoplus_{x\notin C_1^n}\mathcal A \widehat{M}'_x.
 \tag{7.20}\label{canonical-kernel-eq}
\]
\end{lemma}

\begin{proof}
Let $x\notin C_1^n$.  Equations
\eqref{standard-image-lattice} and
\eqref{positive-canonical-triangularity} give
\[
 \pi(\widehat{M}'_x)\in vL^+_{\rm st}.
\]
This vector is bar invariant by Lemma~\ref{bar-compatible-quotient}, so
Lemma~\ref{web-bar-lemma} gives $\pi(\widehat{M}'_x)=0$.

For $c\in C_1^n$, triangularity and
\eqref{standard-image-lattice} show that $\pi(\widehat{M}'_c)\in L^+_{\rm st}$.
It is bar invariant, so its coefficients in the bar-invariant web basis are
constant integers.  Moreover,
\[
 \pi(\widehat{M}'_c)\equiv e_{T_c}\pmod{vL^+_{\rm st}}.
\]
Thus the transition matrix from
$\{\pi(\widehat{M}'_c):c\in C_1^n\}$ to $\mathcal W_r$ is the specialization at
$v=0$ of the unitriangular standard-to-web matrix.  It is unitriangular over
$\mathbb Z$, and the vectors $\pi(\widehat{M}'_c)$ form an $\mathcal A$-basis of
$S^{(r,r)}$.  Since $\{\widehat{M}'_x:x\in\cF_n\}$ is a basis of $\widehat{\mathcal M}$, the
displayed kernel formula follows.
\end{proof}

\begin{theorem}\label{C1-cell-thm}
For every even $n$, the molecule $C_1^n$ is a cell of $\Gamma^\m_n$.
\end{theorem}

\begin{proof}
By Proposition~\ref{specht-quotient} and
Lemma~\ref{canonical-kernel}, the span of
$\{\widehat{M}'_x:x\notin C_1^n\}$ is an $\mathcal H_n$-submodule of $\widehat{\mathcal M}$.
The action matrices on the positive canonical basis $\{\widehat{M}'_x\}$ are the
transposes of the action matrices defining $\Gamma^\m_n$
\cite[Theorem~3.26]{Marberg}.  Transposing the invariant-subspace statement
shows that the span of
$\{\underline M_c:c\in C_1^n\}$ is an $\mathcal H_n$-submodule.
Consequently $C_1^n$ is a union of cells of $\Gamma^\m_n$.

Theorem~\ref{D1-rec-thm} shows that $C_1^n$ is one molecule, and every
molecule is contained in a unique cell.  A set which is both one molecule and
a union of complete cells must equal its unique containing cell.  Hence
$C_1^n$ is a cell.
\end{proof}
\subsection*{Acknowledgements}
The authors declare that the data supporting the findings of this study are available within the paper.

\subsection*{Conflict of Interest}
The authors declare that they have no conflict of interest.

\end{document}